\documentclass[11pt]{article}

\usepackage{fullpage}
\usepackage{amsmath,amsthm,amsfonts,amssymb,mathtools}
\usepackage{upgreek}
\usepackage{hyperref}
\usepackage{caption}
\usepackage{color}
\usepackage[noadjust]{cite}

\hypersetup{hidelinks} 

\theoremstyle{definition}
\newtheorem{definition}{Definition}[section]
\newtheorem{remark}[definition]{Remark}

\newtheorem{example}[definition]{Example}

\theoremstyle{plain}
\newtheorem{theorem}[definition]{Theorem}
\newtheorem{lemma}[definition]{Lemma}
\newtheorem{corollary}[definition]{Corollary}
\newtheorem{proposition}[definition]{Proposition}

\newcommand{\gauss}[2]{\genfrac{[}{]}{0pt}{}{#1}{#2}}

\def\MatXR{\operatorname{Mat}_{X}(\mathbb{R})}
\def\MatF{\operatorname{Mat}_{d+1}(\mathbb{F})}

\def\<{\langle}
\def\>{\rangle}

\begin{document}

\title{\bfseries Geometric parameter sequences for $Q$-polynomial distance-regular graphs}

\author{
	Jae-Ho Lee
	\thanks{\parbox[t]{0.94\textwidth}{
	Department of Mathematics and Statistics, University of North Florida,
	Jacksonville, FL 32224, USA, Email: \texttt{jaeho.lee@unf.edu}
	}}
	\quad
	Jongyook Park
	\thanks{\parbox[t]{0.94\textwidth}{
	Department of Mathematics, Kyungpook National University,
	Daegu 41566, Republic of Korea, Email: \texttt{jongyook@knu.ac.kr} (Corresponding author)
	}}
	\quad
	Ian Seong
	\thanks{\parbox[t]{0.94\textwidth}{
	Department of Mathematics, Williams College,
	Williamstown, MA 01267, USA, Email: \texttt{is11@williams.edu}
	}}
}

\date{}

\maketitle

\begingroup
\renewcommand{\thefootnote}{}
\footnotetext{All authors contributed equally to this paper.}
\addtocounter{footnote}{-1}
\endgroup

\begin{abstract}
Let $\Gamma$ be a distance-regular graph with diameter $D \geq 3$, and let $\theta$ be a nonzero eigenvalue of $\Gamma$ distinct from its valency.
Motivated by the geometric parameter sequences of geometric distance-regular graphs, we introduce analogous sequences defined recursively from the intersection numbers of $\Gamma$ and $\theta$, without assuming a geometric structure.
We show that the recursion is well defined if and only if all terms of the standard sequence corresponding to $\theta$ are nonzero and consecutive terms are distinct.
In particular, the geometric parameter sequences associated with the smallest eigenvalue always exist.
Suppose that $\Gamma$ is $Q$-polynomial with respect to an eigenvalue ordering $\theta_0,\theta_1, \ldots,\theta_D$.
Let $\{\varphi_i\}_{i=1}^{D}$ and $\{\phi_i\}_{i=1}^{D}$ be the first and second split sequences of the primary Leonard system, respectively, and let $\beta$ denote its fundamental parameter.
We prove that the geometric parameter sequences associated with $\theta_D$ exist if and only if $\varphi_i\neq\phi_i$ for every $1\leq i\leq D$, and if and only if $\beta\neq-2$ or $D$ is even.
When the sequences exist, we obtain closed formulas in terms of the primary parameter array.
We introduce the split ratio sequence $\rho_i=\varphi_i/\phi_i$.
Using this sequence, we obtain a bidiagonal factorization of $\mathcal{B}-\theta_D I$, where $\mathcal{B}$ is the intersection matrix of $\Gamma$.
This factorization remains valid even when the geometric parameter sequences associated with $\theta_D$ do not exist.
If $\Gamma$ is geometric and $\theta_D$ is its smallest eigenvalue, then the geometric parameter sequences associated with $\theta_D$ coincide with those of $\Gamma$.
\end{abstract}



\section{Introduction}
Let $\Gamma$ denote a distance-regular graph with diameter $D \geq 3$, valency $k$, and smallest eigenvalue $\theta_{\min}$.
Every clique of $\Gamma$ has size at most $1 - k/\theta_{\min}$ \cite[Proposition~4.4.6]{BCN}.
A clique attaining this bound is called a Delsarte clique.
The graph $\Gamma$ is called geometric whenever there exists a collection $\mathcal{C}$ of Delsarte cliques such that every edge of $\Gamma$ is contained in a unique member of $\mathcal{C}$ \cite{1993Godsil}.
In this case, $\mathcal{C}$ is called a geometric structure on $\Gamma$.
Examples of geometric distance-regular graphs include the Hamming, Johnson, Grassmann, dual polar, and bilinear forms graphs; see \cite[Chapter~9]{BCN}.

\smallskip
Geometric distance-regular graphs have been studied through their Delsarte cliques and associated incidence structures.
Bang, Hiraki, and Koolen \cite{2007BHK} related the intersection numbers of geometric distance-regular graphs to parameters associated with their Delsarte cliques.
For each fixed integer $m\geq 2$, Koolen and Bang \cite{2010KoolenBang} proved that there are only finitely many non-geometric distance-regular graphs with diameter at least three, intersection number $c_2\geq 2$, and smallest eigenvalue at least $-m$.
Metsch \cite{1995Metsch,1999Metsch} used cliques and incidence structures to characterize Grassmann and bilinear forms graphs in certain parameter ranges.
Building on Metsch's work, Lv and Koolen \cite{2026LvKoolen} established sufficient conditions for geometricity and improved the parameter bound for characterizing Grassmann and bilinear forms graphs.
Moreover, Koolen, Lv, and Gavrilyuk \cite{2026KoolenLvGavrilyuk} refined Metsch's clique-based approach to characterize the Grassmann graphs $J_2(2D+3,D)$ for $D\geq 3$.

\smallskip
Suppose that $\Gamma$ is geometric, and fix a geometric structure $\mathcal{C}$.
Associated with $\mathcal{C}$ are two sequences $\{\psi_i\}_{i=0}^{D-1}$ and $\{\tau_i\}_{i=0}^{D}$.
For $0 \leq i \leq D-1$, if a vertex $x$ is at distance $i$ from a clique $C \in \mathcal{C}$, then $\psi_i$ is the number of vertices of $C$ at distance $i$ from $x$.
For $1 \leq i \leq D$, if vertices $x,y$ are at distance $i$, then $\tau_i$ is the number of cliques $C \in \mathcal{C}$ containing $y$ and at distance $i-1$ from $x$.
Both counts depend only on $i$; see \cite[Section~4]{2007BHK}.
We set $\tau_0=0$.
The relations in \cite[Proposition~4.2]{2007BHK} determine $\{\psi_i\}_{i=0}^{D-1}$ and $\{\tau_i\}_{i=0}^{D}$ recursively from the intersection numbers of $\Gamma$ and $\theta_{\min}$.
In particular, both sequences are independent of the choice of $\mathcal{C}$.
We call $\{\psi_i\}_{i=0}^{D-1}$ and $\{\tau_i\}_{i=0}^{D}$ the geometric parameter sequences of $\Gamma$.
In this paper, motivated by the recursive relations, we study when analogous parameter sequences can be defined from the intersection numbers of $\Gamma$ and an eigenvalue without assuming that $\Gamma$ is geometric.

\smallskip
We now summarize the results of this paper.
Let $\theta$ denote a nonzero eigenvalue of $\Gamma$ distinct from $k$, and set $\omega_\theta=1-k/\theta$.
Starting with $\widehat{\psi}_0(\theta)=1$ and $\widehat{\tau}_0(\theta)=0$, we recursively define
$$
	\widehat{\tau}_i(\theta)=\frac{c_i}{\widehat{\psi}_{i-1}(\theta)} \quad (1\leq i\leq D),
	\qquad
	\widehat{\psi}_i(\theta)=\omega_\theta+\frac{b_i}{\theta+\widehat{\tau}_i(\theta)} \quad (1\leq i\leq D-1),
$$
whenever the denominators are nonzero.
Here $b_i$ and $c_i$ are intersection numbers of $\Gamma$.
When this recursion is well defined, we call $\{\widehat{\psi}_i(\theta)\}_{i=0}^{D-1}$ and $\{\widehat{\tau}_i(\theta)\}_{i=0}^{D}$ the geometric parameter sequences associated with $\theta$.
Let $\{u_i(\theta)\}_{i=0}^{D}$ be the standard sequence corresponding to $\theta$. 
We show that the recursion is well defined if and only if $u_i(\theta)\neq 0$ and $u_i(\theta)-u_{i-1}(\theta)\neq 0$ for every $1\leq i\leq D$. 
In particular, the geometric parameter sequences associated with $\theta_{\min}$ always exist.

\smallskip
Now assume that $\Gamma$ is $Q$-polynomial with respect to an eigenvalue ordering $\theta_0,\theta_1,\ldots,\theta_D$.
The primary module for the Terwilliger algebra of $\Gamma$ carries a Leonard system \cite[Lemmas~16.1,~16.2]{2022TerDRG}, which we call the primary Leonard system.
Let $\{\varphi_i\}_{i=1}^{D}$ and $\{\phi_i\}_{i=1}^{D}$ denote the first and second split sequences of the primary Leonard system, respectively, and let $\beta$ be its fundamental parameter.
Our main result shows that the geometric parameter sequences associated with $\theta_D$ exist if and only if $\varphi_i\neq\phi_i$ for every $1\leq i\leq D$, and if and only if $\beta\neq-2$ or $D$ is even; see Theorem~\ref{thm:closed-form}. 
When the sequences exist, we give closed formulas for them in terms of the parameter array of the primary Leonard system.

\smallskip
For $1 \leq i \leq D$, define $\rho_i = \varphi_i/\phi_i$.
We call $\{\rho_i\}_{i=1}^{D}$ the split ratio sequence.
We show that, when the geometric parameter sequences associated with $\theta_D$ exist, they can be expressed in terms of the split ratios.
We also obtain a bidiagonal factorization of $\mathcal{B}-\theta_D I$ in terms of the split ratios, where $\mathcal{B}$ is the intersection matrix of $\Gamma$.
The factorization remains valid even when the geometric parameter sequences associated with $\theta_D$ do not exist.

\smallskip
Suppose further that $\Gamma$ is geometric and $\theta_D = \theta_{\min}$. 
Then the geometric parameter sequences associated with $\theta_D$ coincide with the geometric parameter sequences of $\Gamma$, so the values given by the closed formulas must be integers.
However, the integrality of the closed-form values does not imply geometricity, as shown in Example~\ref{ex:doob}.
Finally, we determine the geometric parameter sequences associated with $\theta_D$ for distance-regular graphs with classical parameters, recovering the geometric parameter sequences of the Hamming, Johnson, Grassmann, dual polar, and bilinear forms graphs.

\smallskip
The paper is organized as follows.
Section~\ref{Sec:preliminaries} reviews Leonard systems, parameter arrays, and recurrent sequences.
Section~\ref{Sec:DRG} reviews distance-regular graphs, the Terwilliger algebra, and the primary Leonard system.
Section~\ref{Sec:geomDRGs} recalls geometric distance-regular graphs and their geometric parameter sequences.
Section~\ref{Sec:geom-parameters} introduces the geometric parameter sequences associated with an eigenvalue and gives a criterion for their well-definedness.
Section~\ref{Sec:Rec-split} develops the recurrence and split sequence relations needed for the $Q$-polynomial setting.
Section~\ref{Sec:closed-form} gives the main existence theorem and closed formulas.
Section~\ref{Sec:split-ratios} introduces the split ratio sequence and develops the associated bidiagonal factorization.
Finally, Section~\ref{Sec:classical-families} treats distance-regular graphs with classical parameters.

\section{Leonard systems and recurrent sequences}
\label{Sec:preliminaries}

In this section, we review some background concerning Leonard systems, parameter arrays, and recurrent sequences.
Let $\mathbb{F}$ denote a field with characteristic zero.
Let $d$ denote a nonnegative integer.

\subsection{Leonard systems and parameter arrays}

Let $\mathsf{V}$ denote a vector space over $\mathbb{F}$ with dimension $d+1$.
Let $\operatorname{End}(\mathsf{V})$ denote the $\mathbb{F}$-algebra consisting of the $\mathbb{F}$-linear maps from $\mathsf{V}$ to $\mathsf{V}$.
Let $\MatF$ denote the $\mathbb{F}$-algebra consisting of the $(d+1)\times(d+1)$ matrices over $\mathbb{F}$ whose rows and columns are indexed by $0,1,\ldots,d$.
Let $\mathbb{F}^{d+1}$ denote the $\mathbb{F}$-vector space consisting of the column vectors of length $d+1$ that have entries in $\mathbb{F}$.
We view $\mathbb{F}^{d+1}$ as a left module for $\MatF$.
An element $\mathsf{A} \in \operatorname{End}(\mathsf{V})$ is said to be \emph{diagonalizable} whenever $\mathsf{V}$ is spanned by the eigenspaces of $\mathsf{A}$.
The element $\mathsf{A}$ is called \emph{multiplicity-free} whenever $\mathsf{A}$ is diagonalizable and each eigenspace of $\mathsf{A}$ has dimension one.
Note that $\mathsf{A}$ is multiplicity-free if and only if $\mathsf{A}$ has $d+1$ mutually distinct eigenvalues in $\mathbb{F}$.
Assume that $\mathsf{A}$ is multiplicity-free.
Let $\{\uptheta_i\}^d_{i=0}$ denote an ordering of eigenvalues of $\mathsf  A$.
For $0 \leq i \leq d$, let $\mathsf  V_i$ denote the eigenspace of $\mathsf A$ corresponding to $\uptheta_i$.
Also, define $\mathsf{E}_i \in \operatorname{End}(\mathsf{V})$ such that  
$$
	(\mathsf{E}_i-\mathsf{I})\mathsf{V}_i=0, \qquad
	\mathsf{E}_i\mathsf{V}_j=0 \quad \text{if} \quad j \ne i  \ (0\leq j \leq d).
$$
Here $\mathsf{I}$ denotes the identity of $\operatorname{End}(\mathsf{V})$.
We observe that (i) $\mathsf{E}_i\mathsf{E}_j = \delta_{ij}\mathsf{E}_i$ $(0 \leq i, j \leq d)$; (ii) $\sum^d_{i=0} \mathsf{E}_i = \mathsf{I}$; (iii) $\mathsf{A}\mathsf{E}_i = \mathsf{E}_i\mathsf{A} = \uptheta_i \mathsf{E}_i$ $(0 \leq i \leq d)$; (iv) $\mathsf{V}_i=\mathsf{E}_i\mathsf{V}$ ($0\leq i \leq d$).
Moreover, 
\begin{equation*}\label{prim idemp}
	\mathsf{E}_i = \prod_{\substack{0 \leq j \leq d \\ j \ne i}}\frac{\mathsf{A}-\uptheta_j \mathsf{I}}{\uptheta_i - \uptheta_j } \qquad (0 \leq i \leq d).
\end{equation*}
We call $\mathsf E_i$ the \emph{primitive idempotent} of $\mathsf A$ for $\uptheta_i$ (or $\mathsf V_i$).

\begin{definition}[cf. {\cite[Definition 1.4]{2001TerLAA}}]\label{Def:LS}
By a \emph{Leonard system} on  $\mathsf V$ we mean a sequence
\begin{equation*}
	\Phi=(\mathsf A; \{\mathsf E_i\}^d_{i=0}; \mathsf A^*; \{\mathsf E^*_i\}^d_{i=0})
\end{equation*}
of elements in $\operatorname{End}(\mathsf V)$ that satisfy the conditions (i)--(v) below.
\begin{itemize}
\setlength{\itemsep}{0pt}
	\item[(i)] Each of $\mathsf A$, $\mathsf A^*$ is multiplicity-free.
	\item[(ii)] $\{\mathsf E_i\}^d_{i=0}$ is an ordering of the primitive idempotents of $\mathsf A$.
	\item[(iii)] $\{\mathsf E^*_i\}^d_{i=0}$ is an ordering of the primitive idempotents of $\mathsf A^*$.
	\item[(iv)] $\mathsf E_i\mathsf A^*\mathsf E_j = \begin{cases}
	0 & \text{if} \quad |i-j|>1; \\
	\ne 0 & \text{if} \quad |i-j|=1
	\end{cases}
	\qquad \qquad (0 \leq i, j \leq d)$.
	\item[(v)] $\mathsf E^*_i\mathsf A\mathsf E^*_j = \begin{cases}
	0 & \text{if} \quad |i-j|>1; \\
	\ne 0 & \text{if} \quad |i-j|=1
	\end{cases}
	\qquad \qquad (0 \leq i, j \leq d)$.
\end{itemize}
The Leonard system $\Phi$ is said to be \emph{over} $\mathbb{F}$ and have \emph{diameter $d$}.
\end{definition}

Let $\Phi = (\mathsf{A}; \{\mathsf{E}_i\}_{i=0}^d; \mathsf{A}^*; \{\mathsf{E}_i^*\}_{i=0}^d)$ denote a Leonard system on $\mathsf V$.
Then the sequence
\begin{equation*}\label{eq:Phi-double-down}
	\Phi^{\Downarrow}=(\mathsf A; \{\mathsf E_{d-i}\}_{i=0}^d; \mathsf A^*; \{\mathsf E^*_i\}_{i=0}^d)
\end{equation*}
is also a Leonard system on $\mathsf V$; see \cite[Section~1]{2001TerLAA}.
For $0 \leq i \leq d$, let $\uptheta_i$ (resp. $\uptheta_i^*$) denote the eigenvalue of $\mathsf A$ (resp. $\mathsf A^*$) associated with $\mathsf E_i$ (resp. $\mathsf E_i^*$).
We call $\{\uptheta_i\}_{i=0}^d$ (resp. $\{\uptheta_i^*\}_{i=0}^d$) the \emph{eigenvalue sequence} (resp. \emph{dual eigenvalue sequence}) of $\Phi$.
By \cite[Theorem~3.2]{2001TerLAA}, there exist nonzero scalars $\{\upvarphi_i\}_{i=1}^d$ and a basis of $\mathsf V$ such that, with respect to this basis, the matrices representing $\mathsf A$ and $\mathsf A^*$ are given as follows:
$$
	\mathsf A:
	\begin{pmatrix}
	\uptheta_0 & & & & {\mathbf0} \\
	1 & \uptheta_1 & & & \\
	& 1 & \uptheta_2 & & \\
	& & \ddots   & \ddots & \\
	{\mathbf 0} & & & 1 & \uptheta_d
	\end{pmatrix},
	\qquad \qquad 
	\mathsf A^*:
	\begin{pmatrix}
	\uptheta_0^* & \upvarphi_1 & & & {\mathbf0} \\
	& \uptheta_1^* & \upvarphi_2 & & \\
	& & \uptheta_2^*  & \ddots & \\
	& & & \ddots & \upvarphi_d \\
	{\mathbf 0} & & & & \uptheta_d^*
	\end{pmatrix}.
$$
We call the sequence $\{\upvarphi_i\}_{i=1}^d$ the \emph{first split sequence} of $\Phi$.
Let $\{\upphi_i\}_{i=1}^d$ denote the first split sequence of $\Phi^\Downarrow$; we call it the \emph{second split sequence} of $\Phi$.
By the \emph{parameter array} of $\Phi$ we mean the sequence 
\begin{equation*}
	(\{\uptheta_i\}^d_{i=0}, \{\uptheta^*_i\}^d_{i=0}, \{\upvarphi_i\}^d_{i=1}, \{\upphi_i\}^d_{i=1}).
\end{equation*}
In \cite{2001TerLAA}, Terwilliger classified Leonard systems up to isomorphism in terms of parameter arrays.

\begin{lemma}[{\cite[Theorem 1.9]{2001TerLAA}}]\label{lem:PA}
Let 
\begin{equation}\label{parameter_array}
	(\{\uptheta_i\}^d_{i=0}, \{\uptheta^*_i\}^d_{i=0}, \{\upvarphi_i\}^d_{i=1}, \{\upphi_i\}^d_{i=1})
\end{equation}
denote a sequence of scalars in $\mathbb{F}$. 
Then there exists a Leonard system $\Phi$ over $\mathbb{F}$ with the parameter array {\rm(\ref{parameter_array})} if and only if the following conditions {\rm(i)--(v)} hold:
\begin{enumerate}
	\item[\rm (i)] $\uptheta_i \ne \uptheta_j,$ \quad $\uptheta^*_i \ne \uptheta^*_j$ \quad if \quad $i \ne j$ \qquad $(0 \leq i,j \leq d)$; 
	\item[\rm(ii)] $\upvarphi_i \ne 0$, \quad $\upphi_i \ne 0$ \qquad $(1 \leq i \leq d)$;
	\item[\rm(iii)] $\upvarphi_i = \upphi_1\sum^{i-1}_{h=0}\frac{\uptheta_h-\uptheta_{d-h}}{\uptheta_0-\uptheta_d}+ (\uptheta^*_i-\uptheta^*_0)(\uptheta_{i-1}-\uptheta_d)$ \qquad $(1 \leq i \leq d)$;
	\item[\rm(iv)] $\upphi_i = \upvarphi_1\sum^{i-1}_{h=0}\frac{\uptheta_h-\uptheta_{d-h}}{\uptheta_0-\uptheta_d}+(\uptheta^*_i-\uptheta^*_0)(\uptheta_{d-i+1}-\uptheta_0)$ \qquad $(1 \leq i \leq d)$;
	\item[\rm(v)] The expressions
	\begin{equation}\label{eq:PA_cond(v)}
		\frac{\uptheta_{i-2}-\uptheta_{i+1}}{\uptheta_{i-1}-\uptheta_i}, \qquad 
		\frac{\uptheta^*_{i-2}-\uptheta^*_{i+1}}{\uptheta^*_{i-1}-\uptheta^*_i}
	\end{equation}
	are equal and independent of $i$ for $2 \leq i \leq d-1$.
\end{enumerate}
Moreover, if {\rm (i)--(v)} hold, then $\Phi$ is unique up to isomorphism
of Leonard systems.
\end{lemma} 

\begin{lemma}[{\cite[Lemma~6.5]{2001TerLAA}}]\label{lem:split-difference-general}
Let $\Phi$ denote a Leonard system with parameter array
$$
	(\{\uptheta_i\}_{i=0}^d, \{\uptheta_i^*\}_{i=0}^d,
	 \{\upvarphi_i\}_{i=1}^d, \{\upphi_i\}_{i=1}^d).
$$
Then, for $1\leq i\leq d$,
\begin{equation}\label{eq:split-difference-general}
	\upvarphi_i - \upphi_i =
	(\uptheta_i^* - \uptheta_{i-1}^*) \sum_{h=0}^{i-1}(\uptheta_h - \uptheta_{d-h}).
\end{equation}
\end{lemma}

\subsection{Recurrent sequences}

In this subsection, we recall some formulas involving recurrent sequences from \cite{2001TerLAA, 2020TerGComb}.
Let $\overline{\mathbb F}$ denote an algebraic closure of $\mathbb F$.
Throughout this subsection, let $\beta$ denote a scalar in $\mathbb{F}$ and let $\{\uptheta_i\}_{i=0}^d$ denote an arbitrary sequence of scalars in $\mathbb{F}$.

\begin{definition}[{\cite[Definition~8.2]{2001TerLAA}}]\label{def:beta-rec}
The sequence $\{\uptheta_i\}_{i=0}^d$ is said to be \emph{$\beta$-recurrent} whenever
\begin{equation*}
	\uptheta_{i-2} - (\beta+1)\uptheta_{i-1} + (\beta+1)\uptheta_i - \uptheta_{i+1} = 0 \qquad (2 \leq i \leq d-1).
\end{equation*}
\end{definition}

By Definition \ref{def:beta-rec}, if $d \leq 2$, then $\{\uptheta_i\}_{i=0}^d$ is $\beta$-recurrent for every $\beta \in \mathbb{F}$.
To avoid this triviality, we assume that $d \geq 3$.

\smallskip
Let $\Phi$ be a Leonard system as in Definition~\ref{Def:LS}.
By Lemma~\ref{lem:PA}(v), there exists a scalar $\beta$ such that both the eigenvalue sequence and the dual eigenvalue sequence of $\Phi$ are $\beta$-recurrent.
More precisely, $\beta+1$ is the common value of \eqref{eq:PA_cond(v)}.
We call $\beta$ the \emph{fundamental parameter of $\Phi$}; see \cite[Definition~1.8]{2015Nomura}.

\smallskip
Assume that $\{\uptheta_i\}_{i=0}^d$ are mutually distinct, and consider the sums
\begin{equation}\label{sum:vartheta}
	\sum^{i-1}_{h=0} \frac{\uptheta_h-\uptheta_{d-h}}{\uptheta_0-\uptheta_d} 
	\qquad
	(0 \leq i \leq d+1).
\end{equation}
We denote the sum in \eqref{sum:vartheta} by $\upvartheta_i$.
We remark that 
$$
	\upvartheta_0 = 0, \qquad 
	\upvartheta_1 = 1, \qquad 
	\upvartheta_d = 1, \qquad 
	\upvartheta_{d+1} = 0.
$$

\begin{lemma}[{\cite[Lemma 10.2]{2001TerLAA}}]\label{lem:theta-sum}
Assume that $\{\uptheta_i\}_{i=0}^d$ are mutually distinct and $\beta$-recurrent.
Then for $0\leq i \leq d+1$ we have the following.
\begin{itemize}
	\item[\rm(i)] Suppose $\beta\ne \pm 2$. Then
	\begin{equation}\label{sum case(i)}
		\sum^{i-1}_{h=0} \frac{\uptheta_h - \uptheta_{d-h}}{\uptheta_0-\uptheta_d} = \frac{(q^i-1)(q^{d-i+1}-1)}{(q-1)(q^d-1)},
	\end{equation}
	where $q\in\overline{\mathbb F}$ satisfies $q+q^{-1}=\beta$.
	
	\item[\rm(ii)] Suppose $\beta=2$. Then
	\begin{equation}\label{sum case(ii)}
		\sum^{i-1}_{h=0} \frac{\uptheta_h-\uptheta_{d-h}}{\uptheta_0-\uptheta_d} = \frac{i(d-i+1)}{d}.
	\end{equation}
	\item[\rm(iii)] Suppose $\beta=-2$ and $d$ odd. Then
	\begin{equation}\label{sum case(iii)}
		\sum^{i-1}_{h=0} \frac{\uptheta_h-\uptheta_{d-h}}{\uptheta_0-\uptheta_d} = 
		\begin{cases}
		0, & \quad \text{if $i$ is even}; \\
		1, & \quad \text{if $i$ is odd}.
		\end{cases}
	\end{equation}
	
	\item[\rm(iv)] Suppose $\beta=-2$ and $d$ even. Then
	\begin{equation}\label{sum case(iv)}
		\sum^{i-1}_{h=0} \frac{\uptheta_h-\uptheta_{d-h}}{\uptheta_0-\uptheta_d} = 
		\begin{cases}
		i/d, & \quad \text{if $i$ is even}; \\
		(d-i+1)/d, & \quad \text{if $i$ is odd}.
		\end{cases}
	\end{equation}
\end{itemize}
\end{lemma}

Assume that $\{\uptheta_i\}_{i=0}^d$ are mutually distinct, and consider the products
\begin{equation}\label{eq:def-upxi}
	\prod_{h=0}^{i-2}
	\frac{\uptheta_i - \uptheta_{h+1}}{\uptheta_{i+1} - \uptheta_h}
	\qquad (1\leq i\leq d-1).
\end{equation}
We denote the product in \eqref{eq:def-upxi} by $\upxi_i$.
Note that $\upxi_1=1$.

\begin{lemma}[{\cite[Lemma~19.5]{2020TerGComb}}]\label{lem:upxi-formulas}
Assume that $\{\uptheta_i\}_{i=0}^d$ are mutually distinct and $\beta$-recurrent.
For $1 \leq i \leq d-1$, the following hold.
\begin{itemize}
	\item[\rm(i)] Suppose $\beta \neq \pm2$. Then
	$$
		\upxi_i = q^{i - 1} \frac{(q - 1)(q^2 - 1)}{(q^i - 1)(q^{i+1} - 1)},
	$$
	where $q\in\overline{\mathbb F}$ satisfies $q+q^{-1}=\beta$.
	\item[\rm(ii)] Suppose $\beta = 2$. Then
	$$
		\upxi_i = \frac{2}{i(i+1)}.
	$$
	\item[\rm(iii)] Suppose $\beta = -2$. Then
	$$
		\upxi_i=
		\begin{cases}
			-2/i, & \text{if $i$ is even},\\
			2/(i+1), & \text{if $i$ is odd}.
		\end{cases}
	$$
\end{itemize}
\end{lemma}

\noindent
The formulas in Lemma~\ref{lem:upxi-formulas} show that $\upxi_i$ depends only on $i$ and $\beta$.
Consequently, $\upxi_i$ is unchanged if $\{\uptheta_i\}_{i=0}^d$ is replaced by another mutually distinct $\beta$-recurrent sequence.

\section{Distance-regular graphs and the primary Leonard system}
\label{Sec:DRG}

In this section, we review some background concerning distance-regular graphs, the Terwilliger algebra, and the primary Leonard system.
For more details, see \cite{BCN, 1992TerJAC, 2022TerDRG}.

\subsection{Distance-regular graphs}
Let $\mathbb{R}$ denote the real number field.
Let $X$ denote a nonempty finite set.
Let $\operatorname{Mat}_X(\mathbb{R})$ denote the $\mathbb{R}$-algebra consisting of matrices with rows and columns indexed by $X$ and all entries in $\mathbb{R}$.
Let $V=\mathbb{R}^X$ denote the vector space over $\mathbb{R}$ consisting of column vectors with coordinates indexed by $X$ and all entries in $\mathbb{R}$.
We view $V$ as a left module for $\MatXR$.
We endow $V$ with the Euclidean inner product $\langle \cdot, \cdot \rangle$ that satisfies $\langle u, v \rangle = u^\top v$ for $u,v \in V$.
For $x \in X$, let $\hat{x}$ denote the vector with $1$ in the $x$-coordinate and $0$ in all other coordinates.
We call $\hat x$ the \emph{characteristic vector of $x$}.
Observe that $\{\hat{x} \mid x \in X\}$ is an orthonormal basis for $V$.
Set $\mathbf 1 = \sum_{x\in X} \hat{x}$, and let $J\in \MatXR$ denote the matrix with all entries equal to $1$.

\smallskip
Next, let $\Gamma$ denote a finite, connected, undirected graph with vertex set $X$ and no loops or multiple edges.
Let $\partial$ denote the path-length distance in $\Gamma$.
The diameter of $\Gamma$ is $D = \max\{ \partial(x,y) \mid x, y \in X\}$.
For $x \in X$ and an integer $i \geq 0$ define $\Gamma_i(x) = \{ y \in X \mid \partial(x,y) = i \}$.
We abbreviate $\Gamma(x) = \Gamma_1(x)$.
For an integer $k \geq 0$ the graph $\Gamma$ is said to be \emph{regular of valency $k$} whenever $|\Gamma(x)|=k$ for every $x \in X$.
The graph $\Gamma$ is said to be \emph{distance-regular} whenever for all integers $0 \leq h, i, j  \leq D$ and all $x, y \in X$ at distance $\partial(x,y)=h$, the number
$$
	p^h_{i,j} = | \Gamma_i(x) \cap \Gamma_j(y) | 
$$
is independent of $x$, $y$ and depends only on $h, i, j$.
The $p^h_{i,j}$ are called the \emph{intersection numbers} of $\Gamma$.
Examples of distance-regular graphs include Johnson graphs, Hamming graphs, and Grassmann graphs \cite[Chapter~9]{BCN}.
Suppose that $\Gamma$ is distance-regular.
By construction $p^h_{i,j} = p^h_{j,i}$ for $0 \leq h,i,j \leq D$.
We abbreviate
\begin{equation*}
	c_i = p^i_{1, i-1} \ (1 \leq i \leq D),\qquad
	a_i = p^i_{1, i} \ (0 \leq i \leq D),\qquad
	b_i = p^i_{1, i+1} \ (0 \leq i \leq D-1).
\end{equation*}
We use the conventions $c_0=0$ and $b_D=0$.
By construction, $c_1 = 1$, $a_0 = 0$, and $c_i > 0$ $(1 \leq i \leq D)$, $b_i>0$ $(0 \leq i \leq D-1)$.
Moreover, $\Gamma$ is regular with valency $k=b_0$ and 
$$
	a_i+b_i+c_i=k \qquad (0 \leq i \leq D).
$$

For the rest of this paper, we assume that $\Gamma$ is distance-regular with $D \geq 3$.
For $0 \leq i \leq D$ define the matrix $A_i \in \operatorname{Mat}_{X}(\mathbb{R})$ by
\begin{equation*}
	(A_i)_{xy} =
	\begin{cases}
	1, & \text{if } \quad \partial(x,y) = i, \\
	0, & \text{if } \quad \partial(x,y) \ne i
	\end{cases}
	\qquad (x,y \in X).
\end{equation*}
We call $A_i$ the \emph{$i^{\rm th}$ distance matrix of $\Gamma$}.
We abbreviate $A=A_1$ and call $A$ the \emph{adjacency matrix} of $\Gamma$.
Observe that $A_0=I$,  $A_i^\top=A_i$ $(0\leq i\leq D)$, and $\sum_{i=0}^D A_i=J$.
By distance-regularity, 
$$
	A_iA_j=\sum_{h=0}^D p_{i,j}^h A_h \qquad (0 \leq i, j \leq D).
$$
Consequently, $A_0,A_1,\ldots,A_D$ form a basis for a commutative subalgebra $M$ of $\MatXR$.
We call $M$ the \emph{Bose-Mesner algebra} of $\Gamma$.
Moreover, $AA_i=b_{i-1}A_{i-1}+a_iA_i+c_{i+1}A_{i+1}$ $(0\leq i\leq D)$, where the terms involving $A_{-1}$ and $A_{D+1}$ are omitted.
It follows that $A_i$ is a polynomial in $A$ for $0\leq i\leq D$.
Note that $A$ is real symmetric, so it is diagonalizable.
By the \emph{eigenvalues of $\Gamma$}, we mean the eigenvalues of $A$.
By \cite[p.~128]{BCN}, the graph $\Gamma$ has exactly $D+1$ distinct eigenvalues, which we denote by $k=\theta_0,\theta_1,\ldots,\theta_D$.
For an eigenvalue $\theta$ of $\Gamma$, define the sequence $\{u_i(\theta)\}_{i=0}^D$ by
$u_0(\theta)=1$, $u_1(\theta)=\theta/k$, and
\begin{equation}\label{eq:standard-sequence}
	c_i u_{i-1}(\theta)+a_i u_i(\theta)+b_i u_{i+1}(\theta)  =\theta u_i(\theta) \qquad (1\leq i\leq D),
\end{equation}
where the term $b_Du_{D+1}(\theta)$ is omitted when $i = D$.
We call $\{u_i(\theta)\}_{i=0}^D$ the \emph{standard sequence} corresponding to $\theta$; see \cite[p.~128]{BCN}.
By \cite[p.~45]{BCN}, the algebra $M$ has a second basis $E_0,E_1,\ldots,E_D$ such that $E_i^\top=E_i$ $(0 \leq i \leq D)$ and 
\begin{equation*}
	\sum_{i=0}^D E_i=I, \qquad E_iE_j=\delta_{ij}E_i \quad (0 \leq i,j \leq D).
\end{equation*}
We order this basis so that $E_i$ is associated with the eigenvalue $\theta_i$.
Thus
$$
	A=\sum_{i=0}^D \theta_iE_i.
$$
We call $E_0,E_1,\ldots,E_D$ the \emph{primitive idempotents of $\Gamma$}.
Note that $E_0=|X|^{-1}J$.

\smallskip
For $B,C \in \operatorname{Mat}_X(\mathbb{R})$, let $B \circ C$ denote the entrywise product of $B$ and $C$. 
Then $M$ is closed under $\circ$ since $A_i \circ A_j = \delta_{i,j} A_i$ for $0 \leq i, j \leq D$.
Therefore, there exist real scalars $q^h_{i,j}$ $(0 \leq h,i,j \leq D)$ such that
$$
	E_i \circ E_j = |X|^{-1} \sum_{h=0}^D q_{i,j}^h E_h \qquad (0 \leq i,j \leq D).
$$
The scalars $q_{i,j}^h$ are called the \emph{Krein parameters} of $\Gamma$.
The ordering $E_0, E_1, \ldots, E_D$ (or $\theta_0, \theta_1, \ldots, \theta_D$) is called \emph{$Q$-polynomial} whenever the following hold for $0 \leq h, i, j \leq D$:
\begin{itemize}
\setlength{\itemsep}{0pt}
\item[(i)] $q_{i,j}^h = 0$ if one of $h,i,j$ is greater than the sum of the other two; 
\item[(ii)] $q^h_{i,j} \ne 0$ if one of $h,i,j$ is equal to the sum of the other two.
\end{itemize}

We say that $\Gamma$ is \emph{$Q$-polynomial} whenever there exists at least one $Q$-polynomial ordering of the primitive idempotents.
In what follows, whenever $\Gamma$ is $Q$-polynomial, we fix such an ordering, denoted by
$$
	E_0,E_1,\ldots,E_D
	\quad (\text{or equivalently, } \theta_0,\theta_1,\ldots,\theta_D).
$$

\subsection{The Terwilliger algebra and the primary Leonard system}

We continue to discuss the graph $\Gamma$.
Fix a vertex $x \in X$.
For $0 \leq i \leq D$, let $E_i^* = E_i^*(x)$ denote the diagonal matrix in $\MatXR$ defined by
\begin{equation*}
	(E_i^*)_{yy}=
	\begin{cases}
	1, & \text{if } \quad \partial(x,y) = i,\\
	0, & \text{if } \quad \partial(x,y) \ne i
	\end{cases}
	\qquad (y \in X).
\end{equation*}
We call $E_i^*$ the $i^{\rm th}$ \emph{dual primitive idempotent of $\Gamma$ with respect to $x$}.
Observe that $(E_i^*)^\top = E_i^*$ $(0 \leq i \leq D)$ and 
\begin{equation}\label{eq:dpidem}
	\sum_{i=0}^D E_i^* = I, \quad
	E_i^*E_j^* = \delta_{i,j} E_i^* \qquad
	(0 \leq i, j \leq D).
\end{equation}
By \eqref{eq:dpidem}, the matrices $\{E_i^*\}_{i=0}^D$ form a basis for a commutative subalgebra $M^*=M^*(x)$ of $\MatXR$.
We call $M^*$ the \emph{dual Bose-Mesner algebra of $\Gamma$ with respect to $x$}.
By construction, $E_i^*V$ is the subspace spanned by $\{\hat{y}\mid y\in \Gamma_i(x)\}$ for $0 \leq i \leq D$.
We call $E_i^*V$ the $i^{\mathrm{th}}$ \emph{subconstituent} of $\Gamma$ with respect to $x$.
Let $T=T(x)$ denote the subalgebra of $\MatXR$ generated by $M$ and $M^*$.
The algebra $T$ is called the \emph{Terwilliger algebra} (or \emph{subconstituent algebra}) of $\Gamma$ with respect to $x$ \cite[Definition~3.3]{1992TerJAC}.

\smallskip
For the rest of this subsection, assume that $\Gamma$ is $Q$-polynomial.
Since $E_1 \in M$, there exist scalars $\{\theta_i^*\}_{i=0}^D$ such that
$$
	E_1 = |X|^{-1} \sum_{i=0}^D \theta_i^* A_i.
$$
We call $\theta_0^*, \theta_1^*, \ldots, \theta_D^*$ the \emph{dual eigenvalues of $\Gamma$}.
By \cite[Lemma~3.11]{1992TerJAC}, the scalars $\{\theta_i^*\}_{i=0}^D$ are mutually distinct.
Define the diagonal matrix $A^*=A^*(x) \in \MatXR$ by
$$
	A^* = \sum_{i=0}^D \theta_i^* E_i^*.
$$
We call $A^*$ the \emph{dual adjacency matrix of $\Gamma$ with respect to $x$}.
Note that $A^*E_i^* = E_i^*A^* = \theta_i^*E_i^*$ for $0 \leq i \leq D$.
By \cite[Lemma~3.11]{1992TerJAC}, $M^*$ is generated by $A^*$.
Therefore, $T$ is generated by $A$ and $A^*$.
By \cite[Lemma~3.2]{1992TerJAC}, the following relations hold in $T$: for $0 \leq i , j \leq D$,
\begin{align*}
	E_i^* A E_j^* &= 0 \qquad \text{if} \quad |i - j| >1, \\
	E_i A^* E_j &= 0 \qquad \text{if} \quad |i - j| >1.
\end{align*}

By a \emph{$T$-module} we mean a subspace $W$ of $V$ such that $TW \subseteq W$.
Note that $V$ is a $T$-module.
A nonzero $T$-module $W$ is said to be \emph{irreducible} whenever it contains no proper nonzero $T$-submodules.
By \cite[Lemma~3.4]{1992TerJAC}, $V$ is an orthogonal direct sum of irreducible $T$-modules.
Let $W$ denote an irreducible $T$-module.
By construction, we have the following orthogonal decompositions:
$$
	W=\sum_{i=0}^D E_i^*W, \qquad
	W=\sum_{i=0}^D E_iW.
$$
The \emph{endpoint} of $W$ is the integer $\min\{i\mid 0\leq i\leq D,\ E_i^*W\neq 0\}$.
The \emph{dual endpoint} of $W$ is the integer $\min\{i\mid 0\leq i\leq D,\ E_iW\neq 0\}$.
The \emph{diameter} of $W$ is $|\{i\mid 0\leq i\leq D,\ E_i^*W\neq 0\}| - 1$.
The \emph{dual diameter} of $W$ is $|\{i\mid 0\leq i\leq D,\ E_iW\neq 0\}| - 1$.
By \cite[Corollary~13.7]{2022TerDRG}, the dual diameter of $W$ is equal to the diameter of $W$.
Let $r$ denote the endpoint of $W$ and let $d$ denote the diameter of $W$.
By \cite[Lemma~3.9]{1992TerJAC}, we have $r+d \leq D$ and $E_i^*W\neq 0$ if and only if $r \leq i \leq r+d$ $(0 \leq i \leq D)$.
The module $W$ is called \emph{thin} whenever $\dim(E_i^*W)\leq1$ for $0\leq i\leq D$, and \emph{dual thin} whenever $\dim(E_iW)\leq1$ for $0\leq i\leq D$.

\begin{example}[cf. {\cite[Lemma~3.6]{1992TerJAC}}]
Consider the vector spaces $M\hat{x}$ and $M^*\mathbf{1}$. 
Then $M\hat{x} = M^*\mathbf{1}$, so this common space is a $T$-module.
Moreover, this $T$-module is thin, dual thin, and irreducible, with endpoint $0$, dual endpoint $0$, and diameter $D$.
This $T$-module is called \emph{primary}.
\end{example}

We now discuss how the primary $T$-module is related to a Leonard system.

\begin{lemma}[cf. {\cite[Lemmas~16.1,~16.2]{2022TerDRG}}]
Let $W$ denote the primary $T$-module. 
Then the sequence
\begin{equation}\label{eq:primaryLS}
	\Phi = (A; \{E_i\}_{i=0}^D; A^*; \{E_i^*\}_{i=0}^D )
\end{equation}
acts on $W$ as a Leonard system.
\end{lemma}
\noindent
The Leonard system $\Phi$ in \eqref{eq:primaryLS} is called \emph{primary}.
Recall the eigenvalues $\{\theta_i\}_{i=0}^D$ and the dual eigenvalues $\{\theta_i^*\}_{i=0}^D$ of $\Gamma$.
Let $\{\varphi_i\}_{i=1}^D$ (resp. $\{\phi_i\}_{i=1}^D$) denote the first split sequence (resp. second split sequence) of $\Phi$.
The sequence
\begin{equation}\label{eq:primaryPA}
	(\{\theta_i\}_{i=0}^D, \{\theta_i^*\}_{i=0}^D, \{\varphi_i\}_{i=1}^D, \{\phi_i\}_{i=1}^D)
\end{equation}
is called the \emph{primary parameter array}.
In what follows, whenever we refer to the \emph{parameter array of $\Gamma$}, we mean the primary parameter array with respect to the fixed $Q$-polynomial ordering.

\begin{lemma}[{\cite[Proposition~19.8]{2020TerGComb}}]\label{lem:intersection-formulas}
For $1\leq i\leq D-1$, let $\xi_i$ denote the scalar defined in \eqref{eq:def-upxi} with $\uptheta_j=\theta_j$ and $d=D$.
The intersection numbers of $\Gamma$ are given by
\begin{alignat}{2}
	b_0&=\frac{\varphi_1}{\theta_1^*-\theta_0^*},\label{eq:b0-PA}\\
	b_i&=
	\frac{\varphi_{i+1}\xi_i(\theta_i^*-\theta_0^*)}
	{(\theta_{i+1}^*-\theta_i^*)
	 (\theta_{i+1}^*-\theta_{i-1}^*)}
	& \qquad & (1\leq i\leq D-1),\label{eq:bi-PA}\\
	c_i&=
	\frac{\phi_i\xi_{D-i}(\theta_i^*-\theta_D^*)}
	{(\theta_{i-1}^*-\theta_i^*)
	 (\theta_{i-1}^*-\theta_{i+1}^*)}
	& \qquad & (1\leq i\leq D-1),\label{eq:ci-PA}\\
	c_D&=\frac{\phi_D}{\theta_{D-1}^*-\theta_D^*}.\label{eq:cD-PA}
\end{alignat}
\end{lemma}

\begin{lemma}\label{lem:end-split}
With reference to the primary parameter array,
\begin{equation}\label{eq:end-split}
	\varphi_1=\theta_0(\theta_1^*-\theta_0^*),
	\qquad
	\phi_1=\theta_D(\theta_1^*-\theta_0^*).
\end{equation}
In particular, $\theta_D\neq 0$.
\end{lemma}

\begin{proof}
By \eqref{eq:b0-PA} and $b_0 = \theta_0$, the first equation in \eqref{eq:end-split} follows.
Applying Lemma~\ref{lem:PA}(iii) with $i=1$ gives
$$
	\varphi_1 - \phi_1 = (\theta_1^* - \theta_0^*)(\theta_0 - \theta_D).
$$
This gives the second equation in \eqref{eq:end-split}. Since $\phi_1\neq 0$ and $\theta_1^* \neq \theta_0^*$, we have $\theta_D \neq 0$.
\end{proof}

\section{Geometric distance-regular graphs}\label{Sec:geomDRGs}
In this section, we recall some background concerning geometric distance-regular graphs and their geometric parameter sequences.
Recall that $\Gamma$ is a distance-regular graph with diameter $D\geq 3$ and valency $k$.
Let $\theta_{\min}$ denote the smallest eigenvalue of $\Gamma$.
Then $\theta_{\min} \leq -1$ by \cite[Corollary~3.5.4]{BCN}.
By \cite[Proposition~4.4.6]{BCN}, every clique in $\Gamma$ has size at most $1 - {k}/{\theta_{\min}}$.
A clique attaining this bound is called a \emph{Delsarte clique}.
The graph $\Gamma$ is called \emph{geometric} whenever there exists a collection $\mathcal{C}$ of Delsarte cliques such that every edge of $\Gamma$ is contained in a unique clique from $\mathcal{C}$ \cite{1993Godsil}. 
In this situation, the collection $\mathcal{C}$ is called a \emph{geometric structure} on $\Gamma$.

\smallskip
Assume that $\Gamma$ is geometric with respect to a fixed geometric structure $\mathcal{C}$. 
For $0 \leq i \leq D-1$, choose $x \in X$ and $C \in \mathcal{C}$ such that $\partial(x,C)=i$, where $\partial(x,C) = \min\{\partial(x,z)\mid z\in C\}$.
The number
\begin{equation*}\label{def:psi_i}
	\psi_i = | C \cap \Gamma_i(x) |
\end{equation*}
is independent of the choice of the pair $x$, $C$ and depends only on $i$ \cite[Section~4]{2007BHK}.
Note that $\psi_0=1$.
We next recall the parameters $\tau_i$ associated with $\mathcal{C}$. 
For $1 \leq i \leq D$, choose $x,y \in X$ such that $\partial(x,y)=i$. 
Consider the number of cliques $C \in \mathcal{C}$ such that $y\in C$ and $\partial(x,C)=i-1$.
By \cite[Lemma~4.1]{2007BHK}, this number is independent of the choice of $x$ and $y$ with $\partial(x,y)=i$, and depends only on $i$. 
We denote this common value by $\tau_i$.
For notational convenience, we set $\tau_0 = 0$.
We call the sequences $\{\psi_i\}_{i=0}^{D-1}$ and $\{\tau_i\}_{i=0}^{D}$ the \emph{geometric parameter sequences} of $\Gamma$ with respect to $\mathcal{C}$.

\begin{lemma}[{\cite[Section~9.3]{vDKT2016}}]\label{lem:vertex-C}
With the above notation, each vertex of $\Gamma$ is contained in exactly $-\theta_{\min}$ cliques from $\mathcal{C}$. 
In particular, $-\theta_{\min}$ is a positive integer.
\end{lemma}

\begin{lemma}[{\cite[Proposition~4.2]{2007BHK}}]\label{lem:Prop4.2BHK}
With the above notation,
\begin{alignat}{2}
	b_i & = - (\theta_{\min} + \tau_i) \left(1 - \frac{k}{\theta_{\min}} - \psi_i \right) &\qquad \quad& (0 \leq i \leq D-1), \label{eq:b_i(1)} \\ 
	c_i & = \tau_i \psi_{i-1} &\qquad \quad & (1 \leq i \leq D). \label{eq:c_i(1)}
\end{alignat}
\end{lemma}

\begin{corollary}\label{cor:psi-tau}
With reference to Lemma~\ref{lem:Prop4.2BHK}, we have
\begin{alignat*}{2}
	\psi_i & = 1 - \frac{k}{\theta_{\min}} + \frac{b_i}{\theta_{\min} + \tau_i} & \qquad \quad & (0 \leq i \leq D-1), \\
	\tau_i & = \frac{c_i}{\psi_{i-1}} & \qquad \quad & (1 \leq i \leq D).
\end{alignat*}
\end{corollary}

\begin{proof}
For $1\leq i\leq D$, the definition of $\psi_{i-1}$ gives $\psi_{i-1}>0$. 
For $0\leq i\leq D-1$, we have $b_i>0$, so \eqref{eq:b_i(1)} implies $\theta_{\min}+\tau_i\neq0$.
The result now follows from \eqref{eq:b_i(1)} and \eqref{eq:c_i(1)}.
\end{proof}

It follows from Corollary~\ref{cor:psi-tau} that the geometric parameter sequences are recursively determined by the intersection numbers $b_i$ and $c_i$ together with $\theta_{\min}$.
In particular, the geometric parameter sequences of $\Gamma$ are independent of the choice of the geometric structure $\mathcal{C}$.

\section{Geometric parameter sequences associated with an eigenvalue}\label{Sec:geom-parameters}
In this section, motivated by the recursive description in Corollary~\ref{cor:psi-tau}, we extend the defining relations for the geometric parameter sequences to a distance-regular graph and a chosen eigenvalue. 
We continue to discuss the distance-regular graph $\Gamma$.
No geometric structure is assumed.
Let $\theta$ be an eigenvalue of $\Gamma$ such that $\theta \neq 0$ and $\theta \neq k$.
Set
\begin{equation*}
	\omega_\theta=1-\frac{k}{\theta}.
\end{equation*}
Define the two sequences $\{\widehat{\psi}_i(\theta)\}_{i=0}^{D-1}$ and $\{\widehat{\tau}_i(\theta)\}_{i=0}^{D}$ recursively as follows.
Set
\begin{equation}\label{def:initial_up}
	\widehat{\psi}_0(\theta)=1,\qquad \widehat{\tau}_0(\theta)=0.
\end{equation}
For $1\leq i\leq D$, define
\begin{equation}\label{def:uptau}
	\widehat{\tau}_i(\theta)=\frac{c_i}{\widehat{\psi}_{i-1}(\theta)},
\end{equation}
provided that $\widehat{\psi}_{i-1}(\theta) \neq 0$.
For $1\leq i\leq D-1$, define
\begin{equation}\label{def:uppsi}
	 \widehat{\psi}_i(\theta)=
	 \omega_\theta+\frac{b_i}{\theta+\widehat{\tau}_i(\theta)},
\end{equation}
provided that $\theta+\widehat{\tau}_i(\theta) \neq 0$.
If the recursive construction \eqref{def:initial_up}--\eqref{def:uppsi} is well defined, then we call $\{\widehat{\psi}_i(\theta)\}_{i=0}^{D-1}$ and $\{\widehat{\tau}_i(\theta)\}_{i=0}^{D}$ the \emph{geometric parameter sequences associated with $\theta$}.

\begin{remark}\label{rem:geometric-parameters}
The above terminology is motivated by the geometric case.
Suppose that $\Gamma$ is geometric and $\theta=\theta_{\min}$.
Then Corollary~\ref{cor:psi-tau} shows that the recursive construction is well defined.
Moreover, $\omega_{\theta}$ is the size of a Delsarte clique, and
$$
	\widehat{\psi}_i(\theta)=\psi_i \qquad (0\leq i\leq D-1),
	\qquad
	\widehat{\tau}_i(\theta)=\tau_i \qquad (0\leq i\leq D).
$$
That is, $\{\widehat{\psi}_i(\theta)\}_{i=0}^{D-1}$ and $\{\widehat{\tau}_i(\theta)\}_{i=0}^{D}$ are the geometric parameter sequences of $\Gamma$.
\end{remark}

We now characterize when the recursive construction is well defined in terms of the standard sequence corresponding to $\theta$.

\begin{lemma}\label{lem:spectral-nondegeneracy}
Let $\{u_i(\theta)\}_{i=0}^D$ denote the standard sequence corresponding to $\theta$.
The recursive construction \eqref{def:initial_up}--\eqref{def:uppsi} is well defined if and only if
\begin{equation}\label{eq:spectral-nondegeneracy}
	 u_i(\theta) \neq 0, \qquad u_i(\theta)-u_{i-1}(\theta) \neq 0 \qquad (1\leq i\leq D).
\end{equation}
When these conditions hold,
\begin{alignat}{2}
	 \widehat{\psi}_i(\theta) & = \omega_\theta \frac{u_{i+1}(\theta)}{u_{i+1}(\theta)-u_i(\theta)}
	 & \qquad &(0\leq i\leq D-1),\label{eq:psi-standard}\\
	 \widehat{\tau}_i(\theta) & = \frac{c_i\bigl(u_i(\theta)-u_{i-1}(\theta)\bigr)} {\omega_\theta u_i(\theta)}
	 & \qquad &(1\leq i\leq D).\label{eq:tau-standard}
\end{alignat}
\end{lemma}

\begin{proof}
Abbreviate $u_i=u_i(\theta)$.
Since $u_0=1$, $u_1=\theta/k$, and $\omega_\theta=1-k/\theta$, we have
$$
	\omega_\theta\frac{u_1}{u_1-u_0}=1=\widehat{\psi}_0(\theta).
$$
For $1\leq i\leq D$, suppose that the recursive construction has reached the step at which $\widehat{\tau}_i(\theta)$ is to be defined, and that \eqref{eq:psi-standard} holds with $i$ replaced by $i-1$.
Then $\widehat{\tau}_i(\theta)$ is defined if and only if $u_i\neq0$.
In this case,
$$
	\widehat{\tau}_i(\theta)
	=\frac{c_i}{\widehat{\psi}_{i-1}(\theta)}
	=\frac{c_i(u_i-u_{i-1})}{\omega_\theta u_i}.
$$
For $1\leq i\leq D-1$, the recurrence \eqref{eq:standard-sequence} and the identity $a_i+b_i+c_i=k$ give
$$
	\theta+\widehat{\tau}_i(\theta)
	=\frac{\theta\omega_\theta u_i+c_i(u_i-u_{i-1})}{\omega_\theta u_i}
	=\frac{b_i(u_{i+1}-u_i)}{\omega_\theta u_i}.
$$
Consequently, $\widehat{\psi}_i(\theta)$ is defined if and only if $u_{i+1}-u_i\neq0$.
In this case,
$$
	\widehat{\psi}_i(\theta)
	=\omega_\theta+\frac{\omega_\theta u_i}{u_{i+1}-u_i}
	=\omega_\theta\frac{u_{i+1}}{u_{i+1}-u_i}.
$$
Since $u_1\neq0$ and $u_1-u_0\neq0$ by the assumptions on $\theta$, proceeding successively from $i=1$ shows that the recursive construction is well defined if and only if \eqref{eq:spectral-nondegeneracy} holds.
In this case, \eqref{eq:psi-standard} and \eqref{eq:tau-standard} hold.
The result follows.
\end{proof}

Consider the tridiagonal matrix of intersection numbers:
\begin{equation}\label{eq:intersection-matrix}
	\mathcal{B}=
	\begin{pmatrix}
	a_0 & b_0 & & & \mathbf{0}\\
	c_1 & a_1 & b_1 & & \\
	& c_2 & a_2 & \ddots & \\
	& & \ddots & \ddots & b_{D-1}\\
	\mathbf{0} & & & c_D & a_D
\end{pmatrix}.
\end{equation}
By \eqref{eq:standard-sequence}, the vector
$(u_0(\theta),u_1(\theta),\ldots,u_D(\theta))^{\top}$ is an eigenvector of $\mathcal{B}$ corresponding to $\theta$.
We next give a bidiagonal factorization of $\mathcal{B} - \theta I$ in terms of the geometric parameter sequences associated with $\theta$.

\begin{lemma}\label{lem:bidiagonal-factorization}
Assume that the recursive construction \eqref{def:initial_up}--\eqref{def:uppsi} is well defined.
Abbreviate $\widehat{\psi}_i=\widehat{\psi}_i(\theta)$ and $\widehat{\tau}_i=\widehat{\tau}_i(\theta)$, and define
$$
	 L_\theta=
	 \begin{pmatrix}
	  -\theta & & & & \mathbf{0}\\
	  \widehat{\tau}_1 & -(\theta+\widehat{\tau}_1) & & & \\
	  & \widehat{\tau}_2 & -(\theta+\widehat{\tau}_2) & & \\
	  & & \ddots & \ddots & \\
	  \mathbf{0} & & & \widehat{\tau}_D & -(\theta+\widehat{\tau}_D)
	 \end{pmatrix}
$$
and
$$
	 U_\theta=
	 \begin{pmatrix}
	  \widehat{\psi}_0 & \omega_\theta-\widehat{\psi}_0 & & & \mathbf{0}\\
	  & \widehat{\psi}_1 & \omega_\theta-\widehat{\psi}_1 & & \\
	  & & \ddots & \ddots & \\
	  & & & \widehat{\psi}_{D-1} & \omega_\theta-\widehat{\psi}_{D-1}\\
	  \mathbf{0} & & & & \omega_\theta
	 \end{pmatrix}.
$$
Then
\begin{equation}\label{eq:bidiagonal-factorization}
	 \mathcal{B}-\theta I=L_\theta U_\theta.
\end{equation}
\end{lemma}

\begin{proof}
By the defining recursion,
$$
	 \widehat{\tau}_i\widehat{\psi}_{i-1} = c_i \ (1\leq i\leq D)
	\qquad
	\text{and}
	\qquad
	(\theta+\widehat{\tau}_i)(\widehat{\psi}_i-\omega_\theta)=b_i \ (1\leq i\leq D-1).
$$
For $i=0$, the second identity follows from $\widehat{\tau}_0=0$, $\widehat{\psi}_0=1$, and $\theta(1-\omega_\theta)=k=b_0$.
Therefore, the subdiagonal and superdiagonal entries of $L_\theta U_\theta$ are $c_i$ and $b_i$, respectively.
For $1\leq i\leq D-1$, the $i$-th diagonal entry of $L_\theta U_\theta$ is
$$
	\widehat{\tau}_i(\omega_\theta-\widehat{\psi}_{i-1}) - (\theta+\widehat{\tau}_i)\widehat{\psi}_i
	= -c_i-b_i-\theta\omega_\theta
	= a_i-\theta.
$$
The same conclusion holds for the $0$-th and $D$-th diagonal entries of $L_\theta U_\theta$ by $a_0=0$, $a_D+c_D=k$, and $\theta\omega_\theta=\theta-k$.
The result follows.
\end{proof}

\begin{corollary}\label{cor:general-end-parameters}
Assume that the recursive construction \eqref{def:initial_up}--\eqref{def:uppsi} is well defined.
Then
\begin{equation}\label{eq:general-end-parameters}
 \widehat{\tau}_D(\theta)=-\theta,
 \qquad
 \widehat{\psi}_{D-1}(\theta)=-\frac{c_D}{\theta}.
\end{equation}
\end{corollary}

\begin{proof}
Since $\theta$ is an eigenvalue of $\mathcal{B}$, $\det(\mathcal{B}-\theta I)=0$.
Every diagonal entry of $U_\theta$ is nonzero, since the recursive construction is well defined and $\omega_\theta\neq0$.
Moreover, every diagonal entry of $L_\theta$ except possibly its $(D,D)$ entry is nonzero.
Therefore, \eqref{eq:bidiagonal-factorization} gives $\theta+\widehat{\tau}_D(\theta)=0$.
The second equation in \eqref{eq:general-end-parameters} follows from $c_D=\widehat{\tau}_D(\theta)\widehat{\psi}_{D-1}(\theta)$.
The result follows.
\end{proof}

\begin{corollary}\label{cor:smallest-eigenvalue-recursion}
Suppose that $\theta=\theta_{\min}$, where $\theta_{\min}$ is the smallest eigenvalue of $\Gamma$.
Then the recursive construction \eqref{def:initial_up}--\eqref{def:uppsi} is well defined.
Moreover, we have
\begin{alignat*}{2}
	0 < & \ \widehat{\psi}_i(\theta) \, < \omega_\theta & \qquad & (0\leq i\leq D-1), \\
	0 < & \ \widehat{\tau}_i(\theta) \ < -\theta & \qquad & (1\leq i\leq D-1), \\
	 & \ \widehat{\tau}_D(\theta) = -\theta.
\end{alignat*}
\end{corollary}

\begin{proof}
The standard sequence corresponding to $\theta$ has exactly
$D$ sign changes; see \cite[Corollary~4.1.2]{BCN}.
Since this sequence has $D+1$ entries, it follows that
$$
	u_{i-1}(\theta)u_i(\theta)<0
	\qquad (1\leq i\leq D).
$$
In particular, the conditions in \eqref{eq:spectral-nondegeneracy} hold, so the recursive construction is well defined by Lemma~\ref{lem:spectral-nondegeneracy}.
The formulas \eqref{eq:psi-standard} and \eqref{eq:tau-standard} give $0 < \widehat{\psi}_i(\theta) < \omega_\theta$ and $\widehat{\tau}_i(\theta) > 0$.
For $1 \leq i \leq D-1$, the defining relation
$$
	b_i =
	(\theta + \widehat{\tau}_i(\theta))
	(\widehat{\psi}_i(\theta) - \omega_\theta)
$$
then gives $\widehat{\tau}_i(\theta)<-\theta$.
The last assertion follows from Corollary~\ref{cor:general-end-parameters}.
\end{proof}

\begin{remark}
When $\Gamma$ is $Q$-polynomial, we will consider the geometric parameter sequences associated with $\theta_D$ for a fixed $Q$-polynomial ordering.
If $\Gamma$ admits more than one $Q$-polynomial ordering, then the resulting sequences may depend on the chosen ordering.
\end{remark}

\section{Recurrence and split sequence relations}\label{Sec:Rec-split}
In this section, we give some recurrence and split sequence relations that will be used in the next section.
We assume that $\Gamma$ is $Q$-polynomial.
Recall that we fixed a $Q$-polynomial ordering $E_0,E_1,\ldots,E_D$.
Recall the parameter array of $\Gamma$:
\begin{equation*}
	(\{\theta_i\}_{i=0}^D,
	 \{\theta_i^*\}_{i=0}^D,
	 \{\varphi_i\}_{i=1}^D,
	 \{\phi_i\}_{i=1}^D).
\end{equation*}
Recall that $\beta$ denotes the fundamental parameter of the primary Leonard system.
For $0 \leq i \leq D+1$, let $\vartheta_i$ denote the scalar defined in \eqref{sum:vartheta}, with $\uptheta_j=\theta_j$ and $d=D$.
By \eqref{eq:split-difference-general} and the definition of $\vartheta_i$,
\begin{equation}\label{eq:split-difference}
	\varphi_i-\phi_i
	=
	(\theta_i^*-\theta_{i-1}^*)
	(\theta_0-\theta_D)\vartheta_i
	\qquad (1\leq i\leq D).
\end{equation}

\begin{lemma}\label{lem:split-nonzero}
The following are equivalent.
\begin{itemize}
\item[\rm(i)] $\varphi_i\neq\phi_i$ for every $1\leq i\leq D$;
\item[\rm(ii)] $\beta\neq-2$ or $D$ is even.
\end{itemize}
Moreover, if $\beta=-2$ and $D$ is odd, then, for $1\leq i\leq D$,
we have $\varphi_i=\phi_i$ if and only if $i$ is even.
\end{lemma}

\begin{proof}
Since the scalars in each of $\{\theta_i\}_{i=0}^D$ and $\{\theta_i^*\}_{i=0}^D$ are mutually distinct, \eqref{eq:split-difference} gives
$$
	\varphi_i \neq \phi_i
	\quad \Longleftrightarrow \quad
	\vartheta_i \neq 0
	\qquad (1\leq i\leq D).
$$
Thus it remains to determine when $\vartheta_i$ is nonzero.
Recall that $\{\theta_i\}_{i=0}^D$ is $\beta$-recurrent.
We use Lemma~\ref{lem:theta-sum} and consider the following four cases.
Suppose first that $\beta \neq \pm 2$, and choose $q$ such that $q + q^{-1} = \beta$.
Since $\{\theta_i\}_{i=0}^D$ is mutually distinct, by \cite[Lemma~9.3(i)]{2001TerLAA} we have $q^j \neq 1$ for $1 \leq j \leq D$.
Therefore, \eqref{sum case(i)} gives $\vartheta_i \neq 0$ for $1 \leq i \leq D$.
Suppose next that $\beta=2$.
By \eqref{sum case(ii)},
$$
	\vartheta_i = \frac{i(D-i+1)}{D} \neq 0 \qquad (1\leq i\leq D).
$$
Suppose next that $\beta=-2$ and $D$ is even.
By \eqref{sum case(iv)}, $\vartheta_i$ is equal to $i/D$ if $i$ is even and $(D-i+1)/D$ if $i$ is odd.
Hence $\vartheta_i\neq0$ for $1\leq i\leq D$.
Finally, suppose that $\beta=-2$ and $D$ is odd.
By \eqref{sum case(iii)},
$$
	\vartheta_i=
	\begin{cases}
		0, & \text{if $i$ is even},\\
		1, & \text{if $i$ is odd}
	\end{cases}
	\qquad (1\leq i\leq D).
$$
The result follows.
\end{proof}

We next relate the standard sequence corresponding to $\theta_D$ to the split sequences of the primary Leonard system.

\begin{lemma}\label{lem:standard-split-product}
Let $\{u_i\}_{i=0}^D$ denote the standard sequence corresponding to $\theta_D$.
Then
\begin{equation}\label{eq:standard-split-product}
	u_i =
	\frac{\phi_1\phi_2\cdots\phi_i}
	     {\varphi_1\varphi_2\cdots\varphi_i}
	\qquad (0 \leq i \leq D).
\end{equation}
In particular, $u_i \neq 0$ for $0 \leq i \leq D$, and
\begin{equation}\label{eq:standard-split-ratio}
	\frac{u_{i-1}}{u_i}
	=
	\frac{\varphi_i}{\phi_i}
	\qquad (1 \leq i \leq D).
\end{equation}
\end{lemma}

\begin{proof}
Equation~\eqref{eq:standard-split-product} follows from \cite[Section~10]{2005TerDCC}.
Since $\varphi_j$ and $\phi_j$ are nonzero, each $u_i$ is nonzero.
Equation~\eqref{eq:standard-split-ratio} now follows from \eqref{eq:standard-split-product}.
\end{proof}

We conclude this section by treating the exceptional case $\beta=-2$ with $D$ odd.

\begin{lemma}\label{lem:exceptional-case}
Suppose that $\beta = -2$ and $D$ is odd.
Then $\theta_D = -1$, and the recursive construction \eqref{def:initial_up}--\eqref{def:uppsi} with $\theta = \theta_D$ fails at the step where $\widehat{\psi}_1(\theta_D)$ is to be defined.
\end{lemma}

\begin{proof}
By Lemma~\ref{lem:split-nonzero}, we have $\varphi_2 = \phi_2$.
Thus \eqref{eq:standard-split-product} gives $u_2 = u_1$, where $\{u_i\}_{i=0}^D$ is the standard sequence corresponding to $\theta_D$.
Since $u_0 = 1$, $u_1 = \theta_D/k$, $c_1 = 1$, and $a_1+b_1 = k-1$, evaluate \eqref{eq:standard-sequence} at $i = 1$ to obtain
$$
	(\theta_D-k)(\theta_D+1) = 0.
$$
Since $\theta_D \neq k$, it follows that $\theta_D = -1$.

\smallskip
The recursion starts with $\widehat{\psi}_0(\theta_D) = 1$, so $\widehat{\tau}_1(\theta_D) = c_1/\widehat{\psi}_0(\theta_D) = 1$.
Consequently,
$$
	\theta_D+\widehat{\tau}_1(\theta_D) = 0,
$$
and \eqref{def:uppsi} does not define $\widehat{\psi}_1(\theta_D)$.
\end{proof}

\noindent
The following example illustrates the exceptional case in Lemma~\ref{lem:exceptional-case}.

\begin{example}\label{ex:odd-graph}
Assume that $D \geq 3$ is odd, and let $\Gamma$ denote the Odd graph $O_{D+1}$; see \cite[Section~9.1D]{BCN}.
The vertices of $\Gamma$ are the $D$-subsets of a fixed $(2D+1)$-set, and two vertices are adjacent whenever they are disjoint.
Since $O_{D+1}$ is the Kneser graph $K(2D+1,D)$, its eigenvalues are
$$
	\theta_i = (-1)^i(D-i+1) \qquad (0 \leq i \leq D);
$$
see \cite[Theorem~9.4.3]{2001GodsilRoyle}.
By \cite[Example~6.1(2)]{1993TerJAC}, $\Gamma$ is $Q$-polynomial with respect to the ordering $\{\theta_i\}_{i=0}^D$.
Thus $\beta = -2$ and $\theta_D = -1$.
By Lemma~\ref{lem:exceptional-case}, the recursive construction associated with $\theta_D$ fails at the step where $\widehat{\psi}_1(\theta_D)$ is to be defined.
\end{example}

\section{Existence and closed formulas}\label{Sec:closed-form}

We continue to discuss the $Q$-polynomial distance-regular graph $\Gamma$.
Recall that we fixed a $Q$-polynomial ordering $\theta_0, \theta_1, \ldots, \theta_D$.
In this section, we characterize when the geometric parameter sequences associated with $\theta_D$ exist and give closed formulas for these sequences.

\smallskip
Set
\begin{equation}\label{eq:omega}
	\omega = \omega_{\theta_D} = 1 - \frac{\theta_0}{\theta_D}.
\end{equation}
By Lemma~\ref{lem:end-split}, $\theta_D \neq 0$, so $\omega$ is well defined.
Since $\theta_D \neq \theta_0 = k$, we have $\omega \neq 0$.
For the rest of the paper, whenever the geometric parameter sequences associated with $\theta_D$ exist, we write
$$
	\widehat{\psi}_i = \widehat{\psi}_i(\theta_D) \quad (0\leq i\leq D-1),
	\qquad
	\widehat{\tau}_i = \widehat{\tau}_i(\theta_D) \quad (0\leq i\leq D).
$$
Let $\{u_i\}_{i=0}^D$ denote the standard sequence corresponding to $\theta_D$, where $u_i=u_i(\theta_D)$.
Recall the parameter array of $\Gamma$ in \eqref{eq:primaryPA}.
Recall that $\beta$ denotes the fundamental parameter of the primary Leonard system.
Recall also that, for $1\leq i\leq D-1$, $\xi_i$ denotes the scalar defined in \eqref{eq:def-upxi} with $\uptheta_j=\theta_j$ and $d=D$.
We now give our main result.

\begin{theorem}\label{thm:closed-form}
With the above notation, the following are equivalent.
\begin{itemize}
	\item[\rm(i)] The geometric parameter sequences associated with $\theta_D$ exist.
	\item[\rm(ii)] $u_i - u_{i-1} \neq 0$ for every $1 \leq i \leq D$.
	\item[\rm(iii)] $\varphi_i \neq \phi_i$ for every $1 \leq i \leq D$.
	\item[\rm(iv)] $\beta \neq -2$ or $D$ is even.
\end{itemize}
When these conditions hold, we have
\begin{alignat}{2}
	\widehat{\psi}_i 
	& = \frac{\theta_0 - \theta_D}{\theta_D}
		\frac{\phi_{i+1}}{\varphi_{i+1} - \phi_{i+1}}
	& \qquad & (0 \leq i \leq D-1),\label{eq:psi-closed}\\
	\widehat{\tau}_0 
	& = 0, \label{eq:tau0-closed}\\
	\widehat{\tau}_i
	& = \frac{\theta_D}{\theta_0 - \theta_D}
		\frac{\xi_{D-i}(\theta_i^* - \theta_D^*)(\varphi_i - \phi_i)}
		{(\theta_{i-1}^* - \theta_i^*)(\theta_{i-1}^* - \theta_{i+1}^*)}
	& \qquad & (1 \leq i \leq D-1),\label{eq:tau-closed}\\
	\widehat{\tau}_D
	&= \frac{\theta_D}{\theta_0 - \theta_D}
		\frac{\varphi_D - \phi_D}{\theta_{D-1}^* - \theta_D^*}.\label{eq:tauD-closed}
\end{alignat}
\end{theorem}

\begin{proof}
By Lemma~\ref{lem:standard-split-product}, $u_i \neq 0$ for $0 \leq i \leq D$.
Therefore, Lemma~\ref{lem:spectral-nondegeneracy} shows that (i) and (ii) are equivalent.
By \eqref{eq:standard-split-ratio},
$$
	u_i - u_{i-1}
	= u_i\left(1 - \frac{\varphi_i}{\phi_i}\right)
	\qquad (1 \leq i \leq D).
$$
Since $u_i$ and $\phi_i$ are nonzero, (ii) and (iii) are equivalent.
The equivalence of (iii) and (iv) follows from Lemma~\ref{lem:split-nonzero}.

\smallskip
Assume that these equivalent conditions hold.
By \eqref{eq:psi-standard}, \eqref{eq:standard-split-ratio}, and \eqref{eq:omega},
$$
	\widehat{\psi}_i
	= \omega \frac{1}{1 - \varphi_{i+1}/\phi_{i+1}}
	= \frac{\theta_0 - \theta_D}{\theta_D}
	\frac{\phi_{i+1}}{\varphi_{i+1} - \phi_{i+1}}
	\qquad (0 \leq i \leq D-1).
$$
This gives \eqref{eq:psi-closed}.
By \eqref{def:initial_up}, $\widehat{\tau}_0 = 0$.
For $1 \leq i \leq D$, equations \eqref{eq:tau-standard}, \eqref{eq:standard-split-ratio}, and \eqref{eq:omega} give
$$
	\widehat{\tau}_i
	= \frac{c_i}{\omega}
	\left(1 - \frac{\varphi_i}{\phi_i}\right)
	= \frac{\theta_D}{\theta_0 - \theta_D}
	\frac{c_i(\varphi_i - \phi_i)}{\phi_i}.
$$
For $1 \leq i \leq D-1$, evaluate this expression using \eqref{eq:ci-PA} to obtain \eqref{eq:tau-closed}.
For $i = D$, evaluate it using \eqref{eq:cD-PA} to obtain \eqref{eq:tauD-closed}.
\end{proof}

We next give a necessary integrality condition for $\Gamma$ to be geometric with $\theta_D=\theta_{\min}$, in terms of $\omega$ and the values of $\widehat{\psi}_i$ and $\widehat{\tau}_i$ given in \eqref{eq:psi-closed}--\eqref{eq:tauD-closed}.

\begin{remark}\label{rem:integrality}
Suppose that $\Gamma$ is geometric and that $\theta_D=\theta_{\min}$.
By Remark~\ref{rem:geometric-parameters}, the geometric parameter sequences associated with $\theta_D$ coincide with the geometric parameter sequences of $\Gamma$.
Hence
$$
	\omega \in \mathbb{Z}_{>0}, \qquad
	\widehat{\psi}_i \in \mathbb{Z} \quad (0\leq i\leq D-1), \qquad
	\widehat{\tau}_i \in \mathbb{Z} \quad (0\leq i\leq D).
$$
Consequently, if $\omega$ is not a positive integer, or if one of the
values in \eqref{eq:psi-closed}--\eqref{eq:tauD-closed} is not an
integer, then $\Gamma$ cannot be geometric with
$\theta_D=\theta_{\min}$.
\end{remark}

The following example shows that the integrality conditions in Remark~\ref{rem:integrality} are not sufficient for geometricity.

\begin{example}\label{ex:doob}
Let $m\geq 1$ and $n\geq 0$, and assume that $2m+n\geq 3$.
The Doob graph $D(m,n)$ is the Cartesian product of $m$ copies of the Shrikhande graph and $n$ copies of $K_4$, and has diameter $D=2m+n$.
It has the same intersection array as the Hamming graph $H(2m+n,4)$; see \cite[Section~9.2B]{BCN}.
The intersection numbers of $D(m,n)$ are $b_i = 3(D - i)$ and $c_i = i$.
Moreover, $D(m,n)$ is $Q$-polynomial with respect to the eigenvalue ordering
$$
	\theta_i = 3D - 4i \qquad (0\leq i\leq D);
$$
see \cite[p.~175]{1997TanJAC}.
In particular, $\theta_D = -D = \theta_{\min}$.
Therefore, the recursive construction gives
$$
	\omega = 4, \qquad
	\widehat{\psi}_i = 1 \quad (0 \leq i\leq D-1),\qquad
	\widehat{\tau}_i = i \quad (0 \leq i\leq D).
$$
Thus all the integrality conditions in Remark~\ref{rem:integrality} are satisfied.

\smallskip
The graph $D(m,n)$ is not geometric; see \cite[Section~4]{2010KoolenBang}.
Indeed, the Delsarte bound is $4$.
An edge arising from a Shrikhande factor is contained in no clique of size $4$, since any clique containing such an edge is contained in the corresponding copy of the Shrikhande graph, whose clique number is $3$.
Thus the integrality conditions do not guarantee the incidence of Delsarte cliques required for geometricity.
\end{example}

\section{Split ratio sequence}\label{Sec:split-ratios}

We continue to discuss the $Q$-polynomial distance-regular graph $\Gamma$.
Recall that we fixed a $Q$-polynomial ordering $\theta_0, \theta_1, \ldots, \theta_D$.
In this section, we introduce the split ratio sequence associated with the parameter array of $\Gamma$.
When the geometric parameter sequences associated with $\theta_D$ exist, we express them in terms of the split ratio sequence.
We also relate the split ratio sequence to the standard sequence corresponding to $\theta_D$ and use this relation to obtain a bidiagonal factorization of $\mathcal B - \theta_D I$.

\begin{definition}\label{def:split-ratio}
Recall the primary Leonard system $\Phi$ and the parameter array of $\Gamma$ in \eqref{eq:primaryPA}.
For $1 \leq i \leq D$, define
\begin{equation}\label{eq:split-ratio}
	\rho_i = \frac{\varphi_i}{\phi_i}.
\end{equation}
We call $\{\rho_i\}_{i=1}^D$ the \emph{split ratio sequence} of $\Phi$.
Note that each $\rho_i$ is well defined and nonzero since $\varphi_i$ and $\phi_i$ are nonzero.
\end{definition}

Let $\{u_i\}_{i=0}^D$ denote the standard sequence corresponding to $\theta_D$.
By \eqref{eq:standard-split-ratio} and \eqref{eq:split-ratio},
\begin{equation}\label{eq:rho-standard}
	\rho_i = \frac{u_{i-1}}{u_i} \qquad (1 \leq i \leq D).
\end{equation}
Recall the scalar $\omega$ from \eqref{eq:omega}.
By \eqref{eq:end-split} and \eqref{eq:split-ratio}, we have
$\rho_1 = \theta_0/\theta_D$.
Hence
\begin{equation}\label{eq:omega-rho}
	\omega = 1 - \rho_1.
\end{equation}
We next express the geometric parameter sequences associated with $\theta_D$ in terms of the split ratio sequence.

\begin{lemma}\label{lem:split-ratios}
Assume that the geometric parameter sequences associated with $\theta_D$ exist.
Then
\begin{align}
	\widehat{\psi}_i &= \frac{1 - \rho_1}{1 - \rho_{i+1}} \qquad (0 \leq i \leq D - 1), \label{eq:psi-rho}\\
	\widehat{\tau}_i &= c_i\frac{1 - \rho_i}{1 - \rho_1} \qquad (1 \leq i \leq D). \label{eq:tau-rho}
\end{align}
\end{lemma}

\begin{proof}
By \eqref{eq:psi-standard}, \eqref{eq:rho-standard}, and \eqref{eq:omega-rho},
$$
	\widehat{\psi}_i
	= \frac{\omega}{1 - \rho_{i+1}}
	= \frac{1 - \rho_1}{1 - \rho_{i+1}}
	\qquad (0 \leq i \leq D - 1).
$$
This gives \eqref{eq:psi-rho}.
Next, by \eqref{eq:tau-standard}, \eqref{eq:rho-standard}, and \eqref{eq:omega-rho},
$$
	\widehat{\tau}_i
	= \frac{c_i}{\omega}(1 - \rho_i)
	= c_i\frac{1 - \rho_i}{1 - \rho_1}
	\qquad (1 \leq i \leq D).
$$
This gives \eqref{eq:tau-rho}.
\end{proof}

Recall the tridiagonal matrix $\mathcal B$ in \eqref{eq:intersection-matrix}.
By Lemma~\ref{lem:bidiagonal-factorization}, if the geometric parameter sequences associated with $\theta_D$ exist, then $\mathcal B - \theta_D I$ has a bidiagonal factorization in terms of these sequences.
We now give a bidiagonal factorization of $\mathcal B - \theta_D I$ in terms of the split ratio sequence.
An important point is that this factorization does not require the geometric parameter sequences associated with $\theta_D$ to exist.
When these sequences do exist, the two factorizations are related by a diagonal rescaling.

\begin{proposition}\label{prop:split-ratio-factorization}
Let $\mathcal B$ denote the intersection matrix in \eqref{eq:intersection-matrix}.
Recall the split ratio sequence $\{\rho_i\}_{i=1}^D$ from \eqref{eq:split-ratio}.
Define
$$
	L_\rho=
	\begin{pmatrix}
	-\dfrac{b_0}{\rho_1} & & & & \mathbf{0} \\[2mm]
	c_1 & -\dfrac{b_1}{\rho_2} & & & \\[2mm]
	& c_2 & \ddots & & \\
	& & \ddots & -\dfrac{b_{D-1}}{\rho_D} & \\[2mm]
	\mathbf{0} & & & c_D & 0
	\end{pmatrix}
$$
and
$$
	R_\rho=
	\begin{pmatrix}
	1 & -\rho_1 & & & \mathbf{0} \\
	& 1 & -\rho_2 & & \\
	& & \ddots & \ddots & \\
	& & & 1 & -\rho_D \\
	\mathbf{0} & & & &1
	\end{pmatrix}.
$$
Then
\begin{equation}\label{eq:split-ratio-factorization}
	\mathcal B - \theta_D I = L_\rho R_\rho.
\end{equation}
\end{proposition}

\begin{proof}
The subdiagonal and superdiagonal entries of $L_\rho R_\rho$ are $c_i$ and $b_i$, respectively.
We now compare the diagonal entries.
Let $\{u_i\}_{i=0}^D$ denote the standard sequence corresponding to $\theta_D$.
For the $0$-th diagonal entry, using \eqref{eq:rho-standard}, $u_0 = 1$, $u_1 = \theta_D/k$,
and $b_0 = k$, we obtain
$$
	-\frac{b_0}{\rho_1} = -b_0\frac{u_1}{u_0} = -\theta_D = a_0 - \theta_D.
$$
For $1 \leq i \leq D - 1$, the $i$-th diagonal entry of $L_\rho R_\rho$ is
\begin{equation}\label{eq:ci-rho-bi}
	-c_i\rho_i - \frac{b_i}{\rho_{i+1}}
	= -c_i\frac{u_{i-1}}{u_i} - b_i\frac{u_{i+1}}{u_i}.
\end{equation}
Evaluate the right-hand side of \eqref{eq:ci-rho-bi} using \eqref{eq:standard-sequence} with $\theta = \theta_D$ and simplify the result to obtain
$$
	-c_i\rho_i - \frac{b_i}{\rho_{i+1}} = a_i - \theta_D.
$$
Finally, the $D$-th diagonal entry is
$$
	-c_D\rho_D = -c_D\frac{u_{D-1}}{u_D} = a_D - \theta_D,
$$
where the last equality again follows from \eqref{eq:standard-sequence} at $i = D$.
Therefore, every entry of $L_\rho R_\rho$ agrees with the corresponding entry of $\mathcal B - \theta_D I$, and \eqref{eq:split-ratio-factorization} follows.
\end{proof}

The next lemma describes the relation between the two bidiagonal factorizations in Proposition~\ref{prop:split-ratio-factorization} and Lemma~\ref{lem:bidiagonal-factorization} when the geometric parameter sequences associated with $\theta_D$ exist.

\begin{lemma}\label{lem:factorization-rescaling}
Suppose that the geometric parameter sequences associated with $\theta_D$ exist.
Recall the matrices $U_{\theta_D}$ and $L_{\theta_D}$ from Lemma~\ref{lem:bidiagonal-factorization} and the matrices $R_\rho$ and $L_\rho$ from Proposition~\ref{prop:split-ratio-factorization}.
Let
$$
	S = \operatorname{diag} (\widehat{\psi}_0,\widehat{\psi}_1,\ldots,\widehat{\psi}_{D-1},\omega).
$$
Then $S$ is invertible, and
\begin{equation}\label{eq:UL-normalization}
	U_{\theta_D} = S R_\rho,
	\qquad
	L_\rho = L_{\theta_D}S.
\end{equation}
\end{lemma}

\begin{proof}
Since each $\widehat{\psi}_i$ is nonzero and $\omega \neq 0$, $S$ is invertible.
By Lemma~\ref{lem:split-ratios},
$$
	\omega - \widehat{\psi}_i
	=
	-\rho_{i+1}\widehat{\psi}_i
	\qquad (0 \leq i \leq D - 1).
$$
Comparing the diagonal and superdiagonal entries of $U_{\theta_D}$ and $S R_\rho$ gives the first equation in \eqref{eq:UL-normalization}.
Using this equation, Lemma~\ref{lem:bidiagonal-factorization}, and Proposition~\ref{prop:split-ratio-factorization}, we have
$$
	L_\rho R_\rho
	= \mathcal B - \theta_D I
	= L_{\theta_D}U_{\theta_D}
	= L_{\theta_D}S R_\rho.
$$
Since $R_\rho$ is invertible, it follows that $L_\rho = L_{\theta_D}S$.
This proves \eqref{eq:UL-normalization}.
\end{proof}

We finish this section with a sign property of the split ratio sequence.

\begin{lemma}\label{lem:split-ratio-sign}
Suppose that $\theta_D = \theta_{\min}$.
Then
\begin{equation*}
	\rho_i < 0
	\qquad (1 \leq i \leq D).
\end{equation*}
\end{lemma}

\begin{proof}
The standard sequence corresponding to $\theta_{\min}$ has exactly $D$ sign changes; see \cite[Corollary~4.1.2]{BCN}.
By Lemma~\ref{lem:standard-split-product}, each $u_i$ is nonzero.
Therefore,
$$
	u_{i-1}u_i < 0 \qquad (1 \leq i \leq D).
$$
The result follows from \eqref{eq:rho-standard}.
\end{proof}

\section{Examples: Classical families of distance-regular graphs}\label{Sec:classical-families}

In this section, we consider distance-regular graphs with classical parameters and determine the geometric parameter sequences associated with an eigenvalue.
We begin by recalling some preliminaries on classical parameters.
For an integer $i \geq 1$, write
$$
	\gauss{i}{1} = 1 + b + b^2 + \cdots + b^{i-1}, \qquad \gauss{0}{1} = 0.
$$
The distance-regular graph $\Gamma$ is said to have \emph{classical parameters} $(D,b,\alpha,\sigma)$ whenever
\begin{alignat}{2}
	b_i
	&=
	\left(\gauss{D}{1} - \gauss{i}{1}\right)
	\left(\sigma - \alpha\gauss{i}{1}\right)
	& \qquad & (0\leq i\leq D-1), \label{eq:classical-bi}\\
	c_i
	&=
	\gauss{i}{1}
	\left(1 + \alpha\gauss{i-1}{1}\right)
	& \qquad & (1\leq i\leq D); \label{eq:classical-ci}
\end{alignat}
see \cite[Section~6.1]{BCN}.
In this case, $b$ is an integer distinct from $0$ and $-1$; see \cite[Proposition~6.2.1]{BCN}.
Moreover, $\Gamma$ is $Q$-polynomial and the eigenvalues in the corresponding $Q$-polynomial ordering are
\begin{equation}\label{eq:classical-eigenvalues}
	\theta_i
	= \gauss{D-i}{1} \left(\sigma - \alpha\gauss{i}{1}\right) - \gauss{i}{1}
	\qquad (0\leq i\leq D);
\end{equation}
see \cite[Corollary~8.4.2]{BCN}.
In particular,
\begin{equation}\label{eq:classical-end-eigenvalues}
	\theta_0 = \sigma\gauss{D}{1}, \qquad \theta_D = -\gauss{D}{1}.
\end{equation}

\begin{lemma}\label{lem:classical-parameters}
Suppose that $\Gamma$ has classical parameters $(D,b,\alpha,\sigma)$.
With respect to the $Q$-polynomial ordering in \eqref{eq:classical-eigenvalues}, the geometric parameter sequences associated with $\theta_D$ exist and are given by
\begin{equation}\label{eq:classical-sequences}
	\widehat{\psi}_i = 1 + \alpha\gauss{i}{1} \qquad (0\leq i\leq D-1),
	\qquad
	\widehat{\tau}_i = \gauss{i}{1} \qquad (0\leq i\leq D).
\end{equation}
\end{lemma}

\begin{proof}
We will show directly that the recursive construction \eqref{def:initial_up}--\eqref{def:uppsi} is well defined.
Since $b$ is an integer distinct from $0$ and $-1$, we have $\gauss{i}{1} \neq 0$ for $1 \leq i \leq D$.
Thus, by \eqref{eq:classical-end-eigenvalues} we have $\theta_D \neq 0$ and
\begin{equation}\label{eq:omega-classical}
	\omega_{\theta_D}
	= 1 - \frac{\theta_0}{\theta_D}
	= 1 + \sigma.
\end{equation}
Moreover, using $\gauss{i}{1} \neq 0$ and $c_i > 0$, \eqref{eq:classical-ci} gives
\begin{equation}\label{eq:classical-factor-nonzero}
	1 + \alpha\gauss{i-1}{1} \neq 0 \qquad (1 \leq i \leq D).
\end{equation}

We now evaluate the recursive construction.
The formulas in \eqref{eq:classical-sequences} hold for $i=0$.
Suppose that
\begin{equation}\label{eq:classical-induction-psi}
	\widehat{\psi}_{i-1}=1+\alpha\gauss{i-1}{1}
\end{equation}
for some $1\leq i\leq D$.
Using \eqref{eq:classical-ci}, \eqref{eq:classical-factor-nonzero}, and \eqref{eq:classical-induction-psi}, we obtain
$$
	\widehat{\tau}_i = \frac{c_i}{\widehat{\psi}_{i-1}} = \gauss{i}{1}.
$$
If $1\leq i\leq D-1$, then
\begin{equation}\label{eq:theta_D+uptau_i}
	\theta_D + \widehat{\tau}_i
	= -\gauss{D}{1} + \gauss{i}{1}
	= -b^i\gauss{D-i}{1}
	\neq 0.
\end{equation}
Using \eqref{eq:classical-bi}, \eqref{eq:omega-classical}, and \eqref{eq:theta_D+uptau_i}, we obtain
$$
	\widehat{\psi}_i
	= \omega_{\theta_D} + \frac{b_i}{\theta_D + \widehat{\tau}_i}
	= 1 + \alpha\gauss{i}{1}.
$$
It follows inductively that the recursive construction is well defined and yields \eqref{eq:classical-sequences}.
\end{proof}

Using Lemma~\ref{lem:classical-parameters}, we now describe the geometric parameter sequences for five classical families of distance-regular graphs: the Hamming, Johnson, Grassmann, dual polar, and bilinear forms graphs.

\begin{example}\label{ex:hamming}
Let $\Gamma=H(D,q)$ denote the Hamming graph, where $q\geq2$; see \cite[Section~9.2]{BCN}.
The graph $\Gamma$ has classical parameters
$$
	(D,b,\alpha,\sigma) = (D,1,0,q-1).
$$
Therefore, by Lemma~\ref{lem:classical-parameters}, the geometric parameter sequences associated with $\theta_D$ are
$$
	\widehat{\psi}_i = 1 \quad (0\leq i\leq D-1),
	\qquad
	\widehat{\tau}_i = i \quad (0\leq i\leq D).
$$
By \eqref{eq:classical-eigenvalues},
$$
	\theta_i = D(q - 1) - qi \qquad (0\leq i\leq D),
$$
so $\theta_D = -D = \theta_{\min}$.
The sets of vertices obtained by fixing all but one coordinate are Delsarte cliques of size $q$.
These cliques form a geometric structure on $\Gamma$.
Hence, by Remark~\ref{rem:geometric-parameters}, the sequences $\{\widehat{\psi}_i\}_{i=0}^{D-1}$ and $\{\widehat{\tau}_i\}_{i=0}^{D}$ coincide with the geometric parameter sequences of $\Gamma$.
\end{example}

\begin{example}\label{ex:johnson}
Let $\Gamma=J(v,D)$ denote the Johnson graph, where $D\leq v/2$; see \cite[Section~9.1]{BCN}.
The graph $\Gamma$ has classical parameters
$$
	(D,b,\alpha,\sigma)=(D,1,1,v-D).
$$
Therefore, by Lemma~\ref{lem:classical-parameters}, the geometric parameter sequences associated with $\theta_D$ are
$$
	\widehat{\psi}_i = i+1 \quad (0\leq i\leq D-1),
	\qquad
	\widehat{\tau}_i = i \quad (0\leq i\leq D).
$$
By \eqref{eq:classical-eigenvalues},
$$
	\theta_i = (D - i)(v - D - i) - i \qquad (0 \leq i \leq D).
$$
Since $D\leq v/2$, these eigenvalues are strictly decreasing, so $\theta_D = -D = \theta_{\min}$.
The sets of all $D$-subsets containing a fixed $(D-1)$-subset are Delsarte cliques of size $v-D+1$.
These cliques form a geometric structure on $\Gamma$.
Hence, by Remark~\ref{rem:geometric-parameters}, the sequences $\{\widehat{\psi}_i\}_{i=0}^{D-1}$ and $\{\widehat{\tau}_i\}_{i=0}^{D}$ coincide with the geometric parameter sequences of $\Gamma$.
\end{example}

\begin{example}\label{ex:grassmann}
Let $\Gamma=J_q(N,D)$ denote the Grassmann graph, where $q$ is a prime power and $N\geq2D$; see \cite[Section~9.3]{BCN}.
The graph $\Gamma$ has classical parameters
$$
	(D,b,\alpha,\sigma) = \left(D,q,q,\gauss{N-D+1}{1} - 1\right).
$$
Therefore, by Lemma~\ref{lem:classical-parameters}, the geometric parameter sequences associated with $\theta_D$ are
$$
	\widehat{\psi}_i = \gauss{i+1}{1} \quad (0\leq i\leq D-1),
	\qquad
	\widehat{\tau}_i = \gauss{i}{1} \quad (0\leq i\leq D).
$$
By \eqref{eq:classical-eigenvalues},
$$
	\theta_i = q^{i+1}\gauss{D-i}{1}\gauss{N-D-i}{1} - \gauss{i}{1} \qquad (0\leq i\leq D).
$$
These eigenvalues are strictly decreasing, so $\theta_D = -\gauss{D}{1} = \theta_{\min}$.
The sets of all $D$-dimensional subspaces containing a fixed $(D-1)$-dimensional subspace are Delsarte cliques of size $\gauss{N-D+1}{1}$.
These cliques form a geometric structure on $\Gamma$.
Hence, by Remark~\ref{rem:geometric-parameters}, the sequences $\{\widehat{\psi}_i\}_{i=0}^{D-1}$ and $\{\widehat{\tau}_i\}_{i=0}^{D}$ coincide with the geometric parameter sequences of $\Gamma$.
\end{example}

\begin{example}\label{ex:dual-polar}
Let $\Gamma$ denote a dual polar graph of diameter $D$ with parameters $q$ and $e$; see \cite[Section~9.4]{BCN}.
The graph $\Gamma$ has classical parameters
$$
	(D,b,\alpha,\sigma)=(D,q,0,q^e).
$$
Therefore, by Lemma~\ref{lem:classical-parameters}, the geometric parameter sequences associated with $\theta_D$ are
$$
	\widehat{\psi}_i = 1 \quad (0\leq i\leq D-1),
	\qquad
	\widehat{\tau}_i = \gauss{i}{1} \quad (0\leq i\leq D).
$$
By \eqref{eq:classical-eigenvalues},
$$
	\theta_i
	=
	q^e\gauss{D-i}{1}-\gauss{i}{1}
	\qquad (0\leq i\leq D).
$$
These eigenvalues are strictly decreasing, so $\theta_D=-\gauss{D}{1}=\theta_{\min}$.
The sets of all maximal totally isotropic subspaces containing a fixed $(D-1)$-dimensional totally isotropic subspace are Delsarte cliques of size $1+q^e$.
These cliques form a geometric structure on $\Gamma$.
Hence, by Remark~\ref{rem:geometric-parameters}, the sequences
$\{\widehat{\psi}_i\}_{i=0}^{D-1}$ and $\{\widehat{\tau}_i\}_{i=0}^{D}$
coincide with the geometric parameter sequences of $\Gamma$.
\end{example}

\begin{example}\label{ex:bilinear}
Let $\Gamma = H_q(D,e)$ denote the bilinear forms graph, where $q$ is a prime power and $D \leq e$; see \cite[Section~9.5]{BCN}.
The vertices of $\Gamma$ are the $D\times e$ matrices over $\mathbb F_q$, where two vertices are adjacent whenever their difference has rank $1$.
The graph $\Gamma$ has classical parameters
$$
	(D,b,\alpha,\sigma) = (D,q,q-1,q^e-1).
$$
Therefore, by Lemma~\ref{lem:classical-parameters}, the geometric parameter sequences associated with $\theta_D$ are
$$
	\widehat{\psi}_i = q^i \quad (0 \leq i \leq D - 1),
	\qquad
	\widehat{\tau}_i = \gauss{i}{1} \quad (0 \leq i \leq D).
$$
By \eqref{eq:classical-eigenvalues},
$$
	\theta_i
	= \frac{q^{D+e-i} - q^D-q^e+1}{q-1}
	\qquad (0 \leq i \leq D).
$$
These eigenvalues are strictly decreasing, so
$\theta_D = -\gauss{D}{1} = \theta_{\min}$.
For a one-dimensional subspace $L$ of $\mathbb F_q^D$, let
$$
	\mathcal M_L
	=
	\{M \in \mathbb F_q^{D\times e} \mid \operatorname{Im}(M) \subseteq L\}.
$$
For any $D\times e$ matrix $N$ over $\mathbb F_q$, the set $N+\mathcal M_L$ is a Delsarte clique of size $q^e$.
These cliques form a geometric structure on $\Gamma$.
Hence, by Remark~\ref{rem:geometric-parameters}, the sequences $\{\widehat{\psi}_i\}_{i=0}^{D-1}$ and $\{\widehat{\tau}_i\}_{i=0}^{D}$ coincide with the geometric parameter sequences of $\Gamma$.
\end{example}

We conclude the paper with the following table, which summarizes the five families in Examples~\ref{ex:hamming}--\ref{ex:bilinear}.
Here $\omega$ denotes the size of a Delsarte clique, and $\{\rho_i\}_{i=1}^D$ denotes the split ratio sequence.
In the table, $0\leq i\leq D-1$ for $\widehat{\psi}_i$ and $0\leq i\leq D$ for $\widehat{\tau}_i$.

\bigskip
\noindent
\begin{minipage}{\textwidth}
\centering
\captionof{table}{Summary of the five classical families in Examples~\ref{ex:hamming}--\ref{ex:bilinear}.}\label{tab:classical-families}
\renewcommand{\arraystretch}{1.5}
\resizebox{\textwidth}{!}{
\begin{tabular}{c|c|c|c|c|c|c}
	\hline
	$\Gamma$ & $(D,b,\alpha,\sigma)$ & $\widehat{\psi}_i$ & $\widehat{\tau}_i$
		& $\theta_D$ & $\omega$ & $\rho_i$ \\
	\hline \hline
	$H(D,q)$
	& $(D,1,0,q-1)$
	& $1$
	& $i$
	& $-D$
	& $q$
	& $1-q$ \\[3mm]

	$J(v,D)$
	& $(D,1,1,v-D)$
	& $i+1$
	& $i$
	& $-D$
	& $v-D+1$
	& $\displaystyle 1-\frac{v-D+1}{i}$ \\[5mm]

	$J_q(N,D)$
	& $\displaystyle \left(D,q,q,\gauss{N-D+1}{1}-1\right)$
	& $\displaystyle \gauss{i+1}{1}$
	& $\displaystyle \gauss{i}{1}$
	& $\displaystyle -\gauss{D}{1}$
	& $\displaystyle \gauss{N-D+1}{1}$
	& $\displaystyle 1-\frac{\gauss{N-D+1}{1}}{\gauss{i}{1}}$ \\[5mm]

	Dual polar
	& $(D,q,0,q^e)$
	& $1$
	& $\displaystyle \gauss{i}{1}$
	& $\displaystyle -\gauss{D}{1}$
	& $1+q^e$
	& $-q^e$ \\[5mm]

	$H_q(D,e)$
	& $(D,q,q-1,q^e-1)$
	& $q^i$
	& $\displaystyle \gauss{i}{1}$
	& $\displaystyle -\gauss{D}{1}$
	& $q^e$
	& $1-q^{e-i+1}$ \\[3mm]
	\hline
\end{tabular}
}
\end{minipage}


\section*{Acknowledgements}
The authors would like to thank Jack H. Koolen for helpful comments and suggestions on the introduction.
Part of this work was carried out during J.-H. Lee's visit to the Department of Mathematics at Kyungpook National University in summer 2026.
J. Park was supported by the National Research Foundation of Korea (NRF) grant funded by the Korea government (MSIT) (RS-2024-00356153).


\end{document}